\documentclass[a4paper,11pt,fleqn]{article}
\usepackage[dvipsnames]{xcolor}
\usepackage{pstricks}
\usepackage{tikz}
\usepackage{enumitem}
\usepackage{amsmath}
\usepackage{amssymb}
\usepackage{amsfonts}
\usepackage{theorem}
\usepackage{mathtools}
\usepackage{pifont}
\usepackage{euscript}
\usepackage{exscale,relsize}
\usepackage{epic,eepic}
\usepackage{multicol}
\usepackage{makeidx}
\usepackage{srcltx} 
\usepackage{charter}
\usepackage{etoolbox}
\usepackage{mathrsfs} 
\usepackage{graphicx}
\usepackage{authblk}
\newcommand{\email}[1]{\href{mailto:#1}{\nolinkurl{#1}}}
\usepackage[normalem]{ulem} 
\allowdisplaybreaks 
\definecolor{labelkey}{rgb}{0,0.08,0.45}
\definecolor{refkey}{rgb}{0,0.6,0.0}
\definecolor{dblue}{HTML}{0455BF}
\definecolor{dgreen}{HTML}{02724A}
\definecolor{Dblue}{HTML}{8602DC}
\definecolor{dred}{HTML}{D90404}
\usepackage{hyperref}
\usepackage{upref}
\hypersetup{colorlinks=true,linktocpage=true,linkcolor=dblue,
citecolor=dgreen,urlcolor=dred}

\renewcommand{\leq}{\ensuremath{\leqslant}}
\renewcommand{\geq}{\ensuremath{\geqslant}}

\newcommand{\scal}[2]{{\langle{{#1}\mid{#2}}\rangle}}
\newcommand{\sscal}[2]{{\big\langle{{#1}\mid{#2}}\big\rangle}}
\newcommand{\pair}[2]{\langle{{#1},{#2}}\rangle} 
 
\newcommand{\Pair}[2]{\big\langle{{#1},{#2}}\big\rangle} 
\newcommand{\menge}[2]{\big\{{#1}~|~{#2}\big\}} 
 
\newcommand{\IDD}{\ensuremath{\operatorname{int}\operatorname{dom}f}}

\newcommand{\Argmin}{\ensuremath{\operatorname{Argmin}}}
\newcommand{\XX}{\ensuremath{{\mathcal X}}}

\newcommand{\Sum}{\ensuremath{\displaystyle\sum}}

\newcommand{\emp}{\ensuremath{\varnothing}}

\newcommand{\Id}{\ensuremath{\operatorname{Id}}}

\newcommand{\RR}{\ensuremath{\mathbb{R}}}
\newcommand{\RP}{\ensuremath{\left[0,{+}\infty\right[}}

\newcommand{\BC}{\ensuremath{\EuScript C}}

\newcommand{\RPP}{\ensuremath{\left]0,{+}\infty\right[}}

\newcommand{\RPX}{\ensuremath{\left[0,{+}\infty\right]}}
\newcommand{\RX}{\ensuremath{\left]{-}\infty,{+}\infty\right]}}

\newcommand{\WC}{\ensuremath{\mathfrak W}}

\newcommand{\NN}{\ensuremath{\mathbb N}}

\newcommand{\intdom}{\ensuremath{\operatorname{int}\operatorname{dom}}}

\newcommand{\weakly}{\ensuremath{\:\rightharpoonup\:}}
\newcommand{\exi}{\ensuremath{\exists\,}}

\newcommand{\ran}{\ensuremath{\operatorname{ran}}}

\newcommand{\zer}{\ensuremath{\operatorname{zer}}}
\newcommand{\pinf}{\ensuremath{{{+}\infty}}}
\newcommand{\minf}{\ensuremath{{{-}\infty}}}

\newcommand{\dom}{\ensuremath{\operatorname{dom}}}
\newcommand{\cdom}{\ensuremath{\operatorname{\overline{dom}}}}

\newcommand{\prox}{\ensuremath{\operatorname{prox}}}

\newcommand{\gra}{\ensuremath{\operatorname{gra}}}

\newcommand{\zeroun}{\ensuremath{\left]0,1\right[}} 
 
\newcommand{\lzeroun}{\ensuremath{\left[0,1\right[}}

\newtheorem{theorem}{Theorem}[section]
\newtheorem{lemma}[theorem]{Lemma}
\newtheorem{corollary}[theorem]{Corollary}
\newtheorem{proposition}[theorem]{Proposition}

\theoremstyle{plain}{\theorembodyfont{\rmfamily}%
}

\theoremstyle{plain}{\theorembodyfont{\rmfamily}%
\newtheorem{example}[theorem]{Example}}
\theoremstyle{plain}{\theorembodyfont{\rmfamily}%
\newtheorem{remark}[theorem]{Remark}}
\theoremstyle{plain}{\theorembodyfont{\rmfamily}%
\newtheorem{algorithm}[theorem]{Algorithm}}
\theoremstyle{plain}{\theorembodyfont{\rmfamily}%
}
\theoremstyle{plain}{\theorembodyfont{\rmfamily}%
\newtheorem{definition}[theorem]{Definition}}
\theoremstyle{plain}{\theorembodyfont{\rmfamily}%
\newtheorem{problem}[theorem]{Problem}}
\theoremstyle{plain}{\theorembodyfont{\rmfamily}%
}

\numberwithin{equation}{section}
\let\setminus\smallsetminus
\let\to\rightarrow
\newcommand*\mute{{\mkern 2mu\cdot\mkern 2mu}}
\begin{document}

\title{\sffamily\huge%
Bregman Forward-Backward Operator Splitting\thanks{Contact 
author: P. L. Combettes, \email{plc@math.ncsu.edu},
phone: +1 (919) 515 2671.
This work was supported by the National 
Science Foundation under grant DMS-1818946.}}

\author{Minh N. B\`ui and Patrick L. Combettes\\
\small
North Carolina State University,
Department of Mathematics,
Raleigh, NC 27695-8205, USA\\
\small \email{mnbui@ncsu.edu}\: and \:\email{plc@math.ncsu.edu}
}
\date{\ttfamily ~}
\maketitle

\centerline{\emph{Dedicated to Terry Rockafellar on the occasion of
his 85th birthday}}

\bigskip

\noindent{\bfseries Abstract.}
We establish the convergence of the forward-backward splitting 
algorithm based on Bregman distances for the sum of two monotone
operators in reflexive Banach spaces. Even in Euclidean spaces,
the convergence of this algorithm has so far been proved only in the
case of minimization problems. The proposed framework features 
Bregman distances that vary over the iterations and
a novel assumption on the single-valued
operator that captures various properties scattered in the
literature. In the minimization setting, we obtain
rates that are sharper than existing ones. 

\medskip


\section{Introduction}

Throughout, $\XX$ is a reflexive real Banach space with topological
dual $\XX^*$. We are concerned with the following monotone
inclusion problem (see Section~\ref{sec:21} for notation and 
definitions).

\begin{problem}
\label{prob:1}
Let $A\colon\XX\to 2^{\XX^*}$ and $B\colon\XX\to 2^{\XX^*}$
be maximally monotone, let $f\in\Gamma_0(\XX)$ be 
essentially smooth, and let $D_f$ be the Bregman
distance associated with $f$. Set $C=(\intdom f)\cap\dom A$ 
and $\mathscr{S}=(\intdom f)\cap\zer(A+B)$. Suppose that
$C\subset\intdom B$, $\mathscr{S}\neq\emp$, 
$B$ is single-valued on $\intdom B$, and there exist
$\delta_1\in\lzeroun$, $\delta_2\in [0,1]$, and 
$\kappa\in\RP$ such that
\begin{multline}
\label{e:1d}
(\forall x\in C)(\forall y\in C)(\forall z\in\mathscr{S})
(\forall y^*\in Ay)(\forall z^*\in Az)\\
\qquad\Pair{y-x}{By-Bz}\leq\kappa D_f(x,y)+
\Pair{y-z}{\delta_1(y^*-z^*)+\delta_2\big(By-Bz\big)}.
\end{multline}
The objective is to 
\begin{equation}
\label{e:prob1}
\text{find}\;\:x\in\intdom f\;\:\text{such that}\;\:0\in Ax+Bx.
\end{equation}
\end{problem}

The central problem \eqref{e:prob1} has extensive connections with
various areas of mathematics and its applications. In Hilbert
spaces, if $B$ is cocoercive, a standard method
for solving \eqref{e:prob1} is the forward-backward algorithm,
which operates with the update 
$x_{n+1}=(\Id+\gamma A)^{-1}(x_n-\gamma Bx_n)$
\cite{Merc79}. This iteration is not applicable beyond
Hilbert spaces since $A$ maps to $\XX^*\neq\XX$.
In addition, there has been a significant body of
work (see, e.g., \cite{Baus17,Sico03,Baus18,Cens97,Joca16,%
Save18,Nguy17,Orti20,Save17})
showing the benefits of replacing standard distances by 
Bregman distances, even in Euclidean spaces.
Given a sequence $(\gamma_n)_{n\in\NN}$ in $\RPP$ and 
a suitable sequence of differentiable convex functions 
$(f_n)_{n\in\NN}$, we propose to solve \eqref{e:prob1} via
the iterative scheme
\begin{equation}
\label{e:2}
(\forall n\in\NN)\quad
x_{n+1}=\big(\nabla f_n+\gamma_nA\big)^{-1}\big(\nabla
f_n(x_n)-\gamma_nBx_n\big),
\end{equation}
which consists of first applying a forward (explicit) step involving
$B$ and then a backward (implicit) step involving $A$.
Let us note that the convergence of such an iterative process has
not yet been established, even in finite-dimensional spaces with a
single function $f_n=f$ and constant parameters
$\gamma_n=\gamma$. Furthermore, the 
novel scheme \eqref{e:2} will be shown to unify and extend 
several iterative methods which have thus far not been brought
together:
\begin{itemize}
\item
The Bregman monotone proximal point algorithm 
\begin{equation}
\label{e:5}
(\forall n\in\NN)\quad
x_{n+1}=\big(\nabla f+\gamma_nA\big)^{-1}\big(\nabla f(x_n)\big)
\end{equation}
of \cite{Sico03} for finding a zero of $A$ in $\intdom f$, where
$f$ is a Legendre function.
\item
The variable metric forward-backward splitting method
\begin{equation}
\label{e:4}
(\forall n\in\NN)\quad
x_{n+1}=\big(U_n+\gamma_nA\big)^{-1}\big(U_nx_n-\gamma_nBx_n\big)
\end{equation}
of \cite{Opti14} for finding a zero of $A+B$ in a Hilbert space,
where $(U_n)_{n\in\NN}$ is a sequence of strongly positive
self-adjoint bounded linear operators.
\item
The splitting method 
\begin{equation}
\label{e:7}
(\forall n\in\NN)\quad
x_{n+1}=\big(\nabla f_n+\gamma_n\partial\varphi\big)^{-1}\big(\nabla
f_n(x_n)-\gamma_n\nabla\psi(x_n)\big)
\end{equation}
of \cite{Nguy17} for finding a minimizer of the sum of the convex 
functions $\varphi$ and $\psi$ in $\intdom f$.
\item
The Renaud--Cohen algorithm
\begin{equation}
\label{e:8}
(\forall n\in\NN)\quad
x_{n+1}=\big(\nabla f+\gamma A\big)^{-1}\big(\nabla
f(x_n)-\gamma Bx_n\big)
\end{equation}
of \cite{Rena97} for finding a zero of $A+B$ in a Hilbert space,
where $f$ is real-valued and strongly convex.
\end{itemize}
Problems which cannot be solved by algorithms
\eqref{e:5}--\eqref{e:8} will be presented in
Example~\ref{ex:56} as well as in Sections~\ref{sec:32} and
\ref{sec:quang}.
New results on the minimization setting will be presented in
Section~\ref{sec:mini}.

The goal of the present paper is to investigate the asymptotic
behavior of \eqref{e:2} under mild conditions on $A$, $B$, and 
$(f_n)_{n\in\NN}$. Let us note that the convergence
proof techniques used in the above four frameworks do not extend
to \eqref{e:2}. For instance, the tools of
\cite{Nguy17} rely heavily on functional inequalities involving
$\varphi$ and $\psi$. On the other hand, the approach of
\cite{Opti14} exploits specific properties of quadratic kernels
in Hilbert spaces, while \cite{Sico03} relies on Bregman
monotonicity properties of the iterates that will no longer hold
in the presence of $B$. Finally, the proofs of \cite{Rena97} 
depend on the strong
convexity of $f$, the underlying Hilbertian structure, and the fact
that the updating equation is governed by a fixed operator. 
Our analysis will not only capture these frameworks but also provide
new methods to solve problems beyond their reach. It hinges on the
theory of Legendre functions and the following new condition, which
will be seen to cover in particular various 
properties such as the cocoercivity assumption used in the standard
forward-backward method in Hilbert spaces \cite{Livre1,Merc79}, as
well as the seemingly unrelated assumptions used in 
\cite{Sico03,Opti14,Nguy17,Rena97} to study
\eqref{e:5}--\eqref{e:8}.

The main result on the convergence of \eqref{e:2} is established 
in Section~\ref{sec:2} for the general scenario described in
Problem~\ref{prob:1}. Section~\ref{sec:3} is dedicated to
special cases and applications. In the context of minimization 
problems, convergence rates on the worst behavior of the method 
are obtained.

\section{Main results}
\label{sec:2}

\subsection{Notation and definitions}
\label{sec:21}
The norm of $\XX$ is denoted by $\|\mute\|$ and the canonical 
pairing between $\XX$ and $\XX^*$ by 
$\pair{{\mkern 1mu\cdot\mkern 2mu}}{{\cdot\mkern 1mu}}$. If $\XX$ 
is Hilbertian, its scalar product is denoted by
$\scal{{\mkern 1mu\cdot}}{{\cdot\mkern 1mu}}$.
The symbols $\weakly$ and $\to$ denote
respectively weak and strong convergence. The set of weak 
sequential cluster points of a sequence 
$(x_n)_{n\in\NN}$ in $\XX$ is denoted by $\WC(x_n)_{n\in\NN}$. 

Let $M\colon\XX\to 2^{\XX^*}$ be a set-valued operator. Then 
$\gra M=\menge{(x,x^*)\in\XX\times\XX^*}{x^*\in Mx}$ is the graph 
of $M$, $\dom M=\menge{x\in\XX}{Mx\neq\emp}$ the domain of $M$, 
$\ran M=\menge{x^*\in\XX^*}{(\exi x\in\XX)\,x^*\in Mx}$
the range of $M$, and $\zer M=\menge{x\in\XX}{0\in Mx}$
the set of zeros of $M$. Moreover, $M$ is monotone if 
\begin{equation}
\big(\forall (x_1,x_1^*)\in\gra M\big)
\big(\forall (x_2,x_2^*)\in\gra M\big)\quad
\pair{x_1-x_2}{x_1^*-x_2^*}\geq 0,
\end{equation}
and maximally monotone if, furthermore, there exists no monotone 
operator from $\XX$ to $2^{\XX^*}$ the graph of which properly 
contains $\gra M$. 

A function $f\colon\XX\to\RX$ is coercive if 
$\lim_{\|x\|\to+\infty}f(x)=+\infty$ and supercoercive if 
$\lim_{\|x\|\to+\infty}f(x)/\|x\|=+\infty$.
$\Gamma_0(\XX)$ is the class of lower semicontinuous convex
functions $f\colon\XX\to\RX$ such that 
$\dom f=\menge{x\in\XX}{f(x)<\pinf}\neq\emp$. 
Now let $f\in\Gamma_0(\XX)$. The conjugate of $f$ 
is the function $f^*\in\Gamma_0(\XX^*)$ defined by
$f^*\colon\XX^*\to\RX\colon x^*\mapsto
\sup_{x\in\XX}(\pair{x}{x^*}-f(x))$, and the subdifferential of $f$ 
is the maximally monotone operator
\begin{equation}
\label{e:subdiff}
\partial f\colon\XX\to 2^{\XX^*}\colon x\mapsto
\menge{x^*\in\XX^*}{(\forall y\in\XX)\,
\pair{y-x}{x^*}+f(x)\leq f(y)}.  
\end{equation}
In addition, $f$ is a Legendre function if it is 
essentially smooth in the sense that $\partial f$ is 
both locally bounded and single-valued on its
domain, and essentially strictly convex in the sense
that $\partial f^*$ is locally bounded on its domain and 
$f$ is strictly convex on every convex subset of $\dom\partial f$
\cite{Ccm01}.
Suppose that $f$ is G\^ateaux differentiable on 
$\IDD\neq\emp$. The Bregman distance associated with $f$ is 
\begin{equation}
\label{e:Bdist}
\begin{aligned}
D_f\colon\XX\times\XX&\to\,[0,\pinf]\\
(x,y)&\mapsto 
\begin{cases}
f(x)-f(y)-\pair{x-y}{\nabla f(y)},&\text{if}\;\;y\in\IDD;\\
\pinf,&\text{otherwise}.
\end{cases}
\end{aligned}
\end{equation}
Given $\alpha\in\RPP$, we define
\begin{equation}
\label{e:Ca}
\BC_\alpha(f)=\menge{g\in\Gamma_0(\XX)}
{\dom g=\dom f,\:g\:\text{is G\^ateaux differentiable on}\:
\intdom f,\:D_g\geq\alpha D_f}.
\end{equation}

\subsection{On condition~\eqref{e:1d}}
The following proposition provides several key illustrations of
the pertinence of \eqref{e:1d} in terms of capturing concrete 
scenarios.

\begin{proposition}
\label{p:7}
Consider the setting of Problem~\ref{prob:1}. Then \eqref{e:1d}
holds in each of the following cases:
\begin{enumerate}
\item
\label{p:7i-}
$\delta_1\in\lzeroun$, $\delta_2=1$, and
$(\forall x\in C)(\forall y\in C)(\forall z\in\mathscr{S})$
$\pair{z-x}{By-Bz}\leq\kappa D_f(x,y)$.
\item
\label{p:7i}
$\delta_1=0$, $\delta_2=1$, and $B=\partial\psi$,
where $\psi\in\Gamma_0(\XX)$ satisfies 
\begin{equation}
\label{e:3p}
(\forall x\in C)(\forall y\in C)(\forall z\in\mathscr{S})\quad
D_\psi(x,y)\leq\kappa D_f(x,y)+D_\psi(x,z)+ D_\psi(z,y).
\end{equation}
\item
\label{p:7v}
$\delta_1=0$, $\delta_2=1$,
and there exists $\psi\in\Gamma_0(\XX)$ such that
$B=\partial\psi$ and $(\forall x\in C)(\forall y\in C)$
$D_\psi(x,y)\leq\kappa D_f(x,y)$.
\item
\label{p:7iii}
$\dom B=\XX$, there exists $\beta\in\RPP$ such that
\begin{equation}
\label{e:0323}
\big(\forall(x,x^*)\in\gra(A+B)\big)\big(\forall(y,y^*)
\in\gra(A+B)\big)\quad\pair{x-y}{x^*-y^*}\geq\beta\|Bx-By\|^2,
\end{equation}
$f$ is Fr\'echet differentiable on $\XX$,
$\nabla f$ is $\alpha$-strongly monotone on $\dom A$
for some $\alpha\in\RPP$,
$\varepsilon\in\left]0,2\beta\right[$,
$\kappa=1/(\alpha(2\beta-\varepsilon))$, and
$\delta_1=\delta_2=(2\beta-\varepsilon)/(2\beta)$.
\item
\label{p:7iv}
$A+B$ is strongly monotone with constant
$\mu\in\RPP$, $B$ is Lipschitzian on $\dom B=\XX$
with constant $\nu\in\RPP$,
$f$ is Fr\'echet differentiable on $\XX$,
$\nabla f$ is $\alpha$-strongly monotone on $\dom A$
for some $\alpha\in\RPP$,
$\varepsilon\in\left]0,2\mu/\nu^2\right[$,
$\kappa=\nu^2/(\alpha(2\mu-\varepsilon\nu^2))$,
and $\delta_1=\delta_2=(2\mu-\varepsilon\nu^2)/(2\mu)$.
\item
\label{p:7ii}
$\dom B=\XX$, $\beta\in\RPP$,
$f$ is Fr\'echet differentiable on $\XX$,
$\nabla f$ is $\alpha$-strongly monotone on $\dom A$
for some $\alpha\in\RPP$,
$\varepsilon\in\left]0,2\beta\right[$,
$\kappa=1/(\alpha(2\beta-\varepsilon))$, $\delta_1=0$,
$\delta_2=(2\beta-\varepsilon)/(2\beta)$,
and one of the following is satisfied:
\begin{enumerate}[label={\rm[\alph*]}]
\item
\label{p:7iia}
$B$ is $\beta$-cocoercive, i.e.,
\begin{equation}
\label{e:9321}
(\forall x\in\XX)(\forall y\in\XX)\quad
\pair{x-y}{Bx-By}\geq\beta\|Bx-By\|^2.
\end{equation}
\item
\label{p:7iib}
$B$ is $\nu$-Lipschitzian for some $\nu\in\RPP$,
and angle bounded with constant $1/(4\beta\nu)$, i.e.,
\begin{equation}
(\forall x\in\XX)(\forall y\in\XX)(\forall z\in\XX)\quad
\pair{y-z}{Bz-Bx}\leq\frac{1}{4\beta\nu}
\pair{x-y}{Bx-By}.
\end{equation}
\item
\label{p:7iic}
$B$ is $(1/\beta)$-Lipschitzian and
there exists $\psi\in\Gamma_0(\XX)$ such that $B=\nabla\psi$.
\end{enumerate}
\end{enumerate}
\end{proposition}
\begin{proof}
\ref{p:7i-}: Let $x\in C$, $y\in C$, and $z\in\mathscr{S}$. Then 
$\pair{y-x}{By-Bz}
=\pair{z-x}{By-Bz}+\pair{y-z}{By-Bz}
\leq\kappa D_f(x,y)+\pair{y-z}{\delta_2(By-Bz)}$.
In view of the monotonicity of $A$, we obtain \eqref{e:1d}.

\ref{p:7i}$\Rightarrow$\ref{p:7i-}:
In the light of \cite[Proposition~4.1.5 and
Corollary~4.2.5]{Borw10}, $\psi$ is G\^ateaux differentiable
on $\intdom\psi$ and $B=\nabla\psi$ on
$\intdom\psi=\intdom B\supset C$.
Hence, we derive from \eqref{e:3p}, \eqref{e:Bdist},
and \cite[Proposition~2.3(ii)]{Sico03} that
\begin{equation}
(\forall x\in C)(\forall y\in C)(\forall z\in\mathscr{S})\quad
\kappa D_f(x,y)
\geq D_\psi(x,y)-D_\psi(x,z)-D_\psi(z,y)=\pair{z-x}{By-Bz}.
\end{equation}

\ref{p:7v}$\Rightarrow$\ref{p:7i}: Clear.

\ref{p:7iii}:
It results from \cite[Theorem~4.2.10]{Borw10} that
$\nabla f$ is continuous. Thus, using the strong monotonicity
of $\nabla f$ on $\dom A$, we obtain
\begin{equation}
\label{e:2085}
(\forall x\in\cdom A)(\forall y\in\cdom A)\quad
\pair{x-y}{\nabla f(x)-\nabla f(y)}\geq\alpha\|x-y\|^2.
\end{equation}
Given $x$ and $y$ in $\cdom A$, define
$\phi\colon\RR\to\RR\colon t\mapsto f(y+t(x-y))$, and observe that,
since $\cdom A$ is convex \cite[Theorem~3.11.12]{Zali02},
$[x,y]\subset\cdom A$ and therefore \eqref{e:2085} yields
\begin{align}
D_f(x,y)
&=\int_0^1\phi'(t)dt-\pair{x-y}{\nabla f(y)}
\nonumber\\
&=\int_0^1\Pair{x-y}{\nabla f(y+t(x-y))-\nabla f(y)}dt
\nonumber\\
&\geq\int_0^1t\alpha\|x-y\|^2dt
\nonumber\\
&=\frac{\alpha}{2}\|x-y\|^2.
\label{e:4730}
\end{align}
In turn, using \eqref{e:0323} and \eqref{e:4730}, we deduce that
\begin{align}
&\hskip -44mm(\forall x\in C)\big(\forall(y,y^*)\in\gra A\big)
\big(\forall(z,z^*)\in\gra A\big)
\nonumber\\
\hskip 12mm\pair{y-x}{By-Bz}
&\leq\Bigg\|\dfrac{y-x}{\sqrt{2\beta-\varepsilon}}\Bigg\|\,
\big\|\sqrt{2\beta-\varepsilon}(By-Bz)\big\|
\nonumber\\
&\leq\frac{\|y-x\|^2}{2(2\beta-\varepsilon)}
+\frac{2\beta-\varepsilon}{2}\|By-Bz\|^2
\label{e:39fh}\\
&\leq\kappa D_f(x,y)+\Pair{y-z}{\delta_1(y^*-z^*)
+\delta_2(By-Bz)}.
\end{align}

\ref{p:7iv}$\Rightarrow$\ref{p:7iii}:
Set $\beta=\mu/\nu^2$. Then
\begin{multline}
\big(\forall(x,x^*)\in\gra(A+B)\big)
\big(\forall(y,y^*)\in\gra(A+B)\big)\\
\pair{x-y}{x^*-y^*}\geq\mu\|x-y\|^2\geq\beta\|Bx-By\|^2.
\end{multline}

\ref{p:7ii}: We consider each case separately.

\ref{p:7iia}:
By arguing as in \eqref{e:4730}, we obtain
$(\forall x\in\dom A)(\forall y\in\dom A)$
$D_f(x,y)\geq(\alpha/2)\|x-y\|^2$.
It thus follows from \eqref{e:39fh} and \eqref{e:9321} that
\begin{align}
&\hskip -44mm(\forall x\in C)\big(\forall(y,y^*)\in\gra A\big)
\big(\forall(z,z^*)\in\gra A\big)
\nonumber\\
\hskip 22mm\pair{y-x}{By-Bz}
&\leq\frac{\|y-x\|^2}{2(2\beta-\varepsilon)}
+\frac{2\beta-\varepsilon}{2}\|By-Bz\|^2
\nonumber\\
&\leq\kappa D_f(x,y)+\Pair{y-z}{\delta_2(By-Bz)}.
\end{align}

\ref{p:7iib}$\Rightarrow$\ref{p:7iia}:
We derive from \cite[Proposition~4]{Bail77} that
$B$ is cocoercive with constant $\beta$.

\ref{p:7iic}$\Rightarrow$\ref{p:7iia}:
This follows from \cite[Corollaire~10]{Bail77}.
\end{proof}

\begin{remark}
\label{r:q}
Condition~\ref{p:7iii} in Proposition~\ref{p:7} first appeared in
\cite{Rena97} and does not seem to have gotten much notice in the
literature. The cocoercivity condition \ref{p:7ii}\ref{p:7iia}
was first used in \cite{Merc79} to prove
the weak convergence of the classical
forward-backward method in Hilbert spaces. 
Finally, in reflexive Banach space minimization
problems, \ref{p:7v} appears in \cite{Nguy17}; see also 
\cite{Baus17} for the Euclidean case. 
\end{remark}

\begin{remark}
\label{r:3}
Condition~\ref{p:7v} is satisfied in particular
when $\XX$ is a Hilbert space,
$f=\|\mute\|^2/2$, $\dom\psi=\XX$, and $\nabla\psi$ is Lipschitzian
\cite[Theorem~18.15]{Livre1}, in which case it is known as the
``descent lemma.''
Condition~\ref{p:7i} can be viewed as an extension 
of this standard descent lemma involving triples $(x,y,z)$ and an
arbitrary Bregman distance $D_f$ in reflexive Banach spaces. 
Let us underline that \ref{p:7i} is more general than \ref{p:7v}. 
Indeed, consider the setting of Problem~\ref{prob:1} with the 
following additional assumptions: $\XX$ is a Hilbert space,
$0\in\intdom f$,
$A$ is the normal cone operator of some self-dual cone $K$,
and there exists a G\^ateaux differentiable convex function
$\psi\colon\XX\to\RR$ such that
\begin{equation}
\label{e:5920}
B=\nabla\psi,\quad\Argmin\psi=\{0\},\quad\text{and}\quad
\nabla\psi(K)\subset K.
\end{equation}
Then $C=(\intdom f)\cap\dom A\subset K$ and $\mathscr{S}=\{0\}$.
Further, for every $x\in C$ and every $y\in C$, \eqref{e:5920}
yields $D_\psi(x,y)-D_\psi(x,0)-D_\psi(0,y)
=\scal{-x}{\nabla\psi(y)-\nabla\psi(0)}
=\scal{-x}{\nabla\psi(y)}\leq 0\leq D_f(x,y)$.
Therefore, \eqref{e:3p} is satisfied. On the other hand, 
\ref{p:7v} does not hold in general. For instance, take
$\XX=\RR$, $K=\RP$, $f=|\mute|^2/2$, and $\psi=|\mute|^{3/2}$.
\end{remark}

\subsection{Forward-backward splitting for monotone inclusions}

The formal setting of the proposed Bregman forward-backward 
splitting method is as follows.

\begin{algorithm}
\label{a:1}
Consider the setting of Problem~\ref{prob:1}. 
Let $\alpha\in\RPP$, let $(\gamma_n)_{n\in\NN}$ be in $\RPP$,
and let $(f_n)_{n\in\NN}$ be in $\BC_\alpha(f)$.
Suppose that the following hold:
\begin{enumerate}[label={\rm[\alph*]}]
\item
\label{a:1a}
$\inf_{n\in\NN}\gamma_n>0$,
$\sup_{n\in\NN}(\kappa\gamma_n)\leq\alpha$, and 
$\sup_{n\in\NN}(\delta_1\gamma_{n+1}/\gamma_n)<1$.
\item
\label{a:1b}
There exists a summable sequence $(\eta_n)_{n\in\NN}$ in $\RP$
such that $(\forall n\in\NN)$ $D_{f_{n+1}}\leq (1+\eta_n)D_{f_n}$.
\item
\label{a:1c}
For every $n\in\NN$, $\nabla f_n$ is strictly monotone on $C$ and 
$(\nabla f_n-\gamma_n{B})(C)\subset\ran(\nabla f_n+\gamma_nA)$.
\end{enumerate}
Take $x_0\in C$ and set $(\forall n\in\NN)$
$x_{n+1}=(\nabla f_n+\gamma_nA)^{-1}(\nabla f_n(x_n)-\gamma_nBx_n)$.
\end{algorithm}

Let us establish basic asymptotic properties of 
Algorithm~\ref{a:1}, starting with the fact that its
viability domain is $C$.

\begin{proposition}
\label{p:1}
Let $(x_n)_{n\in\NN}$ be a sequence generated by Algorithm~\ref{a:1}
and let $z\in\mathscr{S}$. Then $(x_n)_{n\in\NN}$ is a
well-defined sequence in $C$ and the following hold:
\begin{enumerate}
\item
\label{p:1i}
$(D_{f_n}(z,x_n))_{n\in\NN}$ converges.
\item
\label{p:1iii}
$\sum_{n\in\NN}(1-\kappa\gamma_n/\alpha)D_{f_n}(x_{n+1},x_n)<\pinf$
and $\sum_{n\in\NN}(1-\kappa\gamma_n/\alpha)D_f(x_{n+1},x_n)<\pinf$.
\item
\label{p:1v}
$\sum_{n\in\NN}\pair{x_{n+1}-z}{\gamma_{n}^{-1}(\nabla f_n(x_n)
-\nabla f_n(x_{n+1}))-Bx_n+Bz}<\pinf$.
\item
\label{p:1vi}
$\sum_{n\in\NN}(1-\delta_2)\pair{x_n-z}{{B}x_n-{B}z}<\pinf$.
\item
\label{p:1vii}
Suppose that one of the following is satisfied:
\begin{enumerate}[label={\rm[\alph*]}]
\item
\label{p:1viia}
$C$ is bounded.
\item
\label{p:1viib}
$f$ is supercoercive.
\item
\label{p:1viic}
$f$ is uniformly convex.
\item
\label{p:1viic-}
$f$ is essentially strictly convex with $\dom f^*$ open and 
$\nabla f^*$ weakly sequentially continuous.
\item
\label{p:1viid}
$\XX$ is finite-dimensional and $\dom f^*$ is open.
\item
\label{p:1viie}
$f$ is essentially strictly convex and $\displaystyle
\rho=\inf_{\substack{x\in\intdom f\\ y\in\intdom f\\ x\neq y}}
\;\frac{D_{f}(x,y)}{D_{f}(y,x)} \in\RPP$.
\end{enumerate}
Then $(x_n)_{n\in\NN}$ is bounded.
\end{enumerate}
\end{proposition}
\begin{proof}
Take $n\in\NN$, and suppose that
$(y^*,y_1)$ and $(y^*,y_2)$ belong to
$\gra(\nabla f_n+\gamma_n A)^{-1}$.
Then $y^*\in(\nabla f_n+\gamma_n A)y_1$ and
$y^*\in(\nabla f_n+\gamma_n A)y_2$. However,
by virtue of condition~\ref{a:1c} in Algorithm~\ref{a:1},
$\nabla f_n+\gamma_n A$ is strictly monotone. Therefore,
since $\pair{y_1-y_2}{y^*-y^*}=0$, we infer that $y_1=y_2$.
Hence
\begin{equation}
\label{e:2751}
(\nabla f_n+\gamma_nA)^{-1}\;\text{is single-valued on}\:
\dom(\nabla f_n+\gamma_nA)^{-1}=\ran(\nabla
f_n+\gamma_nA).
\end{equation}
Moreover, it follows from \cite[Proposition~4.2.2]{Borw10} and 
\eqref{e:Ca} that
\begin{equation}
\label{e:0841}
\ran(\nabla f_n+\gamma_nA)^{-1}=\dom\nabla f_n\cap\dom A
=(\intdom f_n)\cap\dom A=C.
\end{equation}
Next, we observe that, since $x_0\in C\subset\intdom B$, 
$\nabla f_0(x_0)-\gamma_0Bx_0$ is a singleton. 
Furthermore, in view of condition~\ref{a:1c} in Algorithm~\ref{a:1},
$\nabla f_0(x_0)-\gamma_0Bx_0\in\ran(\nabla
f_0+\gamma_0A)$. We thus deduce from \eqref{e:2751} that
$x_1=(\nabla f_0+\gamma_0A)^{-1}(\nabla f_0(x_0)-\gamma_0Bx_0)$
is uniquely defined. In addition, \eqref{e:0841} yields
$x_1\in\ran(\nabla f_0+\gamma_0A)^{-1}=C$. The conclusion
that $(x_n)_{n\in\NN}$ is a well-defined sequence in $C$
follows by invoking these facts inductively.

\ref{p:1i}--\ref{p:1vi}:
Condition~\ref{a:1a} in Algorithm~\ref{a:1} entails that there 
exists $\varepsilon\in\zeroun$ such that
\begin{equation}
\label{e:7152}
\delta_1\gamma_{n+1}\leq(1-\varepsilon)\gamma_n.
\end{equation}
Now take $x_0^*\in Ax_0$ and set
\begin{equation}
\label{e:4650}
\begin{cases}
x_{n+1}^*=\gamma_n^{-1}\big(\nabla f_n(x_n)-\nabla
f_n(x_{n+1})\big)-Bx_n\\
\Delta_n=D_{f_n}(z,x_n)+\delta_1\gamma_n\pair{x_n-z}{x_n^*+Bz}\\
\theta_n=(1-\kappa\gamma_n/\alpha)D_{f_n}(x_{n+1},x_n)
\\
\qquad\;+\varepsilon\gamma_n\pair{x_{n+1}-z}{x_{n+1}^*+Bz}
+(1-\delta_2)\gamma_n\pair{x_n-z}{Bx_n-Bz}.
\end{cases}
\end{equation}
In view of \eqref{e:4650},
\begin{equation}
\label{e:1010}
(x_{n+1},x_{n+1}^*)\in\gra A.
\end{equation}
In turn, since $(z,-Bz)\in\gra A$ and $A$ is monotone,
\begin{equation}
\label{e:8832}
\pair{x_{n+1}-z}{x_{n+1}^*+Bz}\geq 0.
\end{equation}
Hence, invoking condition~\ref{a:1a} in Algorithm~\ref{a:1}
and the monotonicity of $B$, we obtain $\theta_n\geq 0$.
Next, since $z\in\intdom f=\intdom f_n$ by \eqref{e:Ca},
we derive from \eqref{e:4650} and
\cite[Proposition~2.3(ii)]{Sico03} that
\begin{align}
0&=\Pair{x_{n+1}-z}{\nabla f_n(x_n)-\nabla
f_n(x_{n+1})-\gamma_nBx_n-\gamma_nx_{n+1}^*}
\nonumber\\
&=\Pair{x_{n+1}-z}{\nabla f_n(x_n)-\nabla f_n(x_{n+1})}
+\gamma_n\pair{z-x_{n+1}}{Bx_n-Bz}-
\gamma_n\pair{x_{n+1}-z}{x_{n+1}^*+Bz}
\nonumber\\
&=D_{f_n}(z,x_n)-D_{f_n}(z,x_{n+1})-D_{f_n}(x_{n+1},x_n)
+\gamma_n\pair{z-x_{n+1}}{Bx_n-Bz}
\nonumber\\
&\quad\;
-\gamma_n\pair{x_{n+1}-z}{x_{n+1}^*+Bz}.
\label{e:2581}
\end{align}
Thus, since $(z,-Bz)\in\gra A$ and $f_n\in\BC_\alpha(f)$,
we infer from \eqref{e:7152}, \eqref{e:8832}, \eqref{e:1010},
and \eqref{e:1d} that
\begin{align}
&D_{f_n}(z,x_{n+1})+\delta_1\gamma_{n+1}
\pair{x_{n+1}-z}{x_{n+1}^*+Bz}
\nonumber\\
&\hskip 5mm
\leq D_{f_n}(z,x_{n+1})+\gamma_n\pair{x_{n+1}-z}{x_{n+1}^*+Bz}
-\varepsilon\gamma_n\pair{x_{n+1}-z}{x_{n+1}^*+Bz}
\nonumber\\
&\hskip 5mm
=D_{f_n}(z,x_n)-D_{f_n}(x_{n+1},x_n)+
\gamma_n\pair{z-x_{n+1}}{Bx_n-Bz}
-\varepsilon\gamma_n\pair{x_{n+1}-z}{x_{n+1}^*+Bz}
\nonumber\\
&\hskip 5mm
=D_{f_n}(z,x_n)-D_{f_n}(x_{n+1},x_n)+
\gamma_n\pair{x_n-x_{n+1}}{Bx_n-Bz}
-\gamma_n\pair{x_n-z}{Bx_n-Bz}
\nonumber\\
&\hskip 5mm
\quad\;-\varepsilon\gamma_n\pair{x_{n+1}-z}{x_{n+1}^*+Bz}
\nonumber\\
&\hskip 5mm
\leq D_{f_n}(z,x_n)-D_{f_n}(x_{n+1},x_n)+\kappa\gamma_n
D_f(x_{n+1},x_n)
+\delta_1\gamma_n\pair{x_n-z}{x_n^*+Bz}
\nonumber\\
&\hskip 5mm
\quad\;+\delta_2\gamma_n\pair{x_n-z}{Bx_n-Bz}
-\gamma_n\pair{x_n-z}{Bx_n-Bz}
-\varepsilon\gamma_n\pair{x_{n+1}-z}{x_{n+1}^*+Bz}
\nonumber\\
&\hskip 5mm
\leq D_{f_n}(z,x_n)+\delta_1\gamma_n\pair{x_n-z}{x_n^*+Bz}
-(1-\kappa\gamma_n/\alpha)D_{f_n}(x_{n+1},x_n)
\nonumber\\
&\hskip 5mm
\quad\;
-\varepsilon\gamma_n\pair{x_{n+1}-z}{x_{n+1}^*+Bz}
-(1-\delta_2)\gamma_n\pair{x_n-z}{Bx_n-Bz}
\nonumber\\
&\hskip 5mm
=\Delta_n-\theta_n.
\end{align}
Consequently, by condition~\ref{a:1b} in Algorithm~\ref{a:1}
and \eqref{e:8832},
\begin{align}
\Delta_{n+1}
&=D_{f_{n+1}}(z,x_{n+1})+\delta_1\gamma_{n+1}
\pair{x_{n+1}-z}{x_{n+1}^*+Bz}
\nonumber\\
&\leq(1+\eta_n)\big(D_{f_n}(z,x_{n+1})+\delta_1\gamma_{n+1}
\pair{x_{n+1}-z}{x_{n+1}^*+Bz}\big)
\nonumber\\
&\leq(1+\eta_n)(\Delta_n-\theta_n)
\nonumber\\
&\leq(1+\eta_n)\Delta_n-\theta_n.
\end{align}
Hence, \cite[Lemma~5.31]{Livre1} asserts that
\begin{equation}
\label{e:1637}
(\Delta_n)_{n\in\NN}\;\text{converges and}\;
\sum_{n\in\NN}\theta_n<\pinf.
\end{equation}
In turn, we infer from \eqref{e:4650} and condition~\ref{a:1a}
in Algorithm~\ref{a:1} that
\begin{equation}
\begin{cases}
\Sum_{n\in\NN}(1-\kappa\gamma_n/\alpha)D_{f_n}(x_{n+1},x_n)<\pinf\\
\Sum_{n\in\NN}\pair{x_{n+1}-z}{x_{n+1}^*+Bz}<\pinf\\
\Sum_{n\in\NN}(1-\delta_2)\pair{x_n-z}{Bx_n-Bz}<\pinf. 
\end{cases}
\end{equation}
Thus, since $(f_n)_{n\in\NN}$ lies in $\BC_\alpha(f)$, we obtain
$\sum_{n\in\NN}(1-\kappa\gamma_n/\alpha)D_f(x_{n+1},x_n)<\pinf$.
It results from \eqref{e:1637} and
\eqref{e:4650} that $(D_{f_n}(z,x_n))_{n\in\NN}$
converges.

\ref{p:1vii}: 
Recall that $(x_n)_{n\in\NN}$ lies in $C$.

\ref{p:1viia}: Clear.

\ref{p:1viib}:
We derive from \ref{p:1i} that $(D_f(z,x_n))_{n\in\NN}$ is bounded.
In turn, \cite[Lemma~7.3(viii)]{Ccm01} asserts that
$(x_n)_{n\in\NN}$ is bounded.

\ref{p:1viic}:
It results from \cite[Theorem~3.5.10]{Zali02} that
there exists a function
$\phi\colon\RP\to\RPX$ that vanishes only at $0$ such that
$\lim_{t\to\pinf}\phi(t)/t\to\pinf$ and
\begin{equation}
\label{e:3872}
(\forall x\in\intdom f)(\forall y\in\dom f)
\quad
\pair{y-x}{\nabla f(x)}+f(x)+\phi\big(\|x-y\|\big)\leq f(y).
\end{equation}
Hence, in the light of \ref{p:1i},
$\sup_{n\in\NN}\phi(\|x_n-z\|)
\leq\sup_{n\in\NN}D_f(z,x_n)
\leq(1/\alpha)\sup_{n\in\NN}D_{f_n}(z,x_n)<\pinf$
and $(x_n)_{n\in\NN}$ is therefore bounded.

\ref{p:1viic-}:
Suppose that there exists a subsequence
$(x_{k_n})_{n\in\NN}$ of $(x_n)_{n\in\NN}$ such that
$\|x_{k_n}\|\to\pinf$.
We deduce from \cite[Lemma~7.3(vii)]{Ccm01} and \ref{p:1i} that
\begin{equation}
\label{e:2917}
\sup_{n\in\NN}D_{f^*}\big(\nabla f(x_n),\nabla f(z)\big)
=\sup_{n\in\NN}D_{f}(z,x_n)
\leq\dfrac{1}{\alpha}\sup_{n\in\NN}D_{f_n}(z,x_n)<\pinf.
\end{equation}
However, $f^*$ is a Legendre function by virtue of
\cite[Corollary~5.5]{Ccm01}
and $\nabla f(z)\in\intdom f^*$
by virtue of \cite[Theorem~5.10]{Ccm01}.
Thus, \cite[Lemma~7.3(v)]{Ccm01}
guarantees that $D_{f^*}(\mute,\nabla f(z))$ is coercive. 
It therefore follows from \eqref{e:2917} that 
$(\nabla f(x_{k_n}))_{n\in\NN}$ is bounded,
and then from the reflexivity of $\XX^*$ that
$\WC(\nabla f(x_{k_n}))_{n\in\NN}\neq\emp$.
In turn, there exist a subsequence $(x_{l_{k_n}})_{n\in\NN}$ of
$(x_{k_n})_{n\in\NN}$ and $x^*\in\XX^*$ such that
$\nabla f(x_{l_{k_n}})\weakly x^*$.
The weak lower semicontinuity of $f^*$
and \eqref{e:2917} yield
$D_{f^*}(x^*,\nabla f(z))\leq\varliminf D_{f^*}
(\nabla f(x_{l_{k_n}}),\nabla f(z))<\pinf$. Therefore
\begin{equation}
\label{e:5079}
\nabla f(x_{l_{k_n}})\weakly x^*\in\dom f^*=\intdom f^*.
\end{equation}
Moreover, \cite[Theorem~5.10]{Ccm01} asserts that
$\nabla f^*(x^*)\in\intdom f$
and $(\forall n\in\NN)\;\nabla f^*\big(\nabla f (x_n)\big)=x_n$.
Hence, \eqref{e:5079}
and the weak sequential continuity of $\nabla f^*$ imply that
$x_{l_{k_n}}=\nabla f^*(\nabla f(x_{l_{k_n}}))
\weakly\nabla f^*(x^*)$. This yields
$\sup_{n\in\NN}\|x_{l_{k_n}}\|<\pinf$ and we reach a contradiction.

\ref{p:1viid}:
A consequence of \cite[Lemma~7.3(ix)]{Ccm01} and \ref{p:1i}.

\ref{p:1viie}:
It results from \cite[Lemma~7.3(v)]{Ccm01} that $D_f(\mute,z)$ is
coercive. In turn, since $\sup_{n\in\NN}D_f(x_n,z)\leq(1/\rho)
\sup_{n\in\NN}D_f(z,x_n)<\pinf$ by \ref{p:1i}, 
$(x_n)_{n\in\NN}$ is bounded.
\end{proof}

As seen in Proposition~\ref{p:1}, by construction, an orbit of
Algorithm~\ref{a:1} lies in $C$ and therefore in $\IDD$. 
Next, we proceed to identify sufficient conditions that
guarantee that their weak sequential cluster points are also in
$\IDD$.

\begin{proposition}
\label{p:2}
Let $(x_n)_{n\in\NN}$ be a sequence generated by Algorithm~\ref{a:1}
and suppose that one of the following holds:
\begin{enumerate}[label={\rm[\alph*]}]
\item
\label{p:2a}
$\cdom f\cap\cdom A\subset\intdom f$.
\item
\label{p:2b}
$f$ is essentially strictly convex with $\dom f^*$ open and 
$\nabla f^*$ weakly sequentially continuous.
\item
\label{p:2c}
$f$ is strictly convex on $\intdom f$ and $\displaystyle
\rho=\inf_{\substack{x\in\intdom f\\ y\in\intdom f\\ x\neq y}}
\;\frac{D_{f}(x,y)}{D_{f}(y,x)}\in\RPP$.
\item
\label{p:2d}
$\XX$ is finite-dimensional.
\end{enumerate}
Then $\WC(x_n)_{n\in\NN}\subset\intdom f$.
\end{proposition}
\begin{proof}
Suppose that $x\in\WC(x_n)_{n\in\NN}$, say $x_{k_n}\weakly x$,
and fix $z\in\mathscr{S}$.

\ref{p:2a}:
Since $\cdom f$ is closed and convex, it is weakly
closed \cite[Corollary~II.6.3.3(i)]{Bour81}. Hence,
since Proposition~\ref{p:1} asserts that 
$(x_n)_{n\in\NN}$ lies in $C\subset\dom f$, we infer that
$\WC(x_n)_{n\in\NN}\subset\cdom f$. Likewise,
since $\cdom A$ is a closed convex set 
\cite[Theorem~3.11.12]{Zali02} and
$(x_n)_{n\in\NN}$ lies in $C\subset\dom A$, we obtain
$\WC(x_n)_{n\in\NN}\subset\cdom A$. Altogether,
$\WC(x_n)_{n\in\NN}\subset\cdom f\cap\cdom A\subset\intdom f$.

\ref{p:2b}:
Using an argument similar to that of the proof of
Proposition~\ref{p:1}\ref{p:1vii}\ref{p:1viic-},
we infer that there exist
a strictly increasing sequence $(l_{k_n})_{n\in\NN}$
in $\NN$ and $x^*\in\intdom f^*$ such that
$x_{l_{k_n}}\weakly\nabla f^*(x^*)$.
Thus, appealing to \cite[Theorem~5.10]{Ccm01}, we conclude that
$x=\nabla f^*(x^*)\in\intdom f$.

\ref{p:2c}:
Proposition~\ref{p:1}\ref{p:1i} and the weak lower semicontinuity
of $D_f(\mute,z)$ yield
\begin{equation}
D_f(x,z)\leq\varliminf D_f(x_{k_n},z)
\leq(1/\rho)\varliminf D_f(z,x_{k_n})
\leq(\alpha\rho)^{-1}\lim D_{f_{k_n}}(z,x_{k_n})<\pinf.
\end{equation}
Thus $x\in\dom f$. We show that $\dom f$ is open.
Suppose that there exists $y\in\dom f\setminus\intdom f$, 
let $(\alpha_n)_{n\in\NN}$ be a sequence in $\zeroun$ such that 
$\alpha_n\to 1$, and set
$(\forall n\in\NN)$ $y_n=\alpha_ny+(1-\alpha_n)z$. Then 
$\{y_n\}_{n\in\NN}\subset\left]y,z\right[
\subset(\intdom f)\setminus\{z\}$
\cite[Proposition~II.2.6.16]{Bour81}. 
Moreover, $y_n\to y$ and, by convexity of $f$, 
$(\forall n\in \NN)$
$D_{f}(y_n,z)\leq\alpha_n(f(y)-f(z)-\pair{y-z}{\nabla f(z)})$.
Hence
\begin{equation}
\label{e:0871}
\varlimsup D_{f}(y_n,z)
\leq f(y)-f(z)-\pair{y-z}{\nabla f(z)}=D_{f}(y,z).
\end{equation}
However, it results from the lower semicontinuity of $f$ that
$\varliminf D_{f}(y_n,z) 
=\varliminf(f(y_n)-f(z))
-\lim\pair{y_n-z}{\nabla f(z)}
\geq f(y)-f(z)
-\pair{y-z}{\nabla f(z)}
=D_{f}(y,z)$.
Hence, \eqref{e:0871} forces
\begin{equation}
\label{e:0590}
\lim D_{f}(y_n,z)=D_{f}(y,z).
\end{equation}
In addition, by convexity of $f$,
$(\forall n\in\NN)\; 
D_{f}(z,y_n)\geq\alpha_n(f(z)-f(y)-\pair{z-y}{\nabla f(y_n)})$.
However,
\cite[Theorem~5.6]{Ccm01}
and the essential smoothness of $f$ entail that 
\begin{equation}
\pair{z-y}{\nabla f(y_n)}=\pair{z-y}{\nabla
f(y+(1-\alpha_n)(z-y))}\to\minf.
\end{equation}
Thus,
\begin{equation}
\label{e:05902}
\pinf
=\lim\Big(\alpha_n\big(f(z)-f(y)
-\pair{z-y}{\nabla f(y_n)}\big)\Big)
\leq\varliminf D_{f}(z,y_n).
\end{equation}
It results from \eqref{e:0590} and \eqref{e:05902} that 
$0<\rho\leq\lim D_{f}(y_n,z)/D_{f}(z,y_n)=0$, so that we reach a 
contradiction. Consequently, $\dom f$ is open and hence
$x\in\dom f=\intdom f$.

\ref{p:2d}:
Proposition~\ref{p:1}\ref{p:1i} ensures that
$(x_{k_n})_{n\in\NN}$ is a sequence in $\intdom f$
such that $(D_f(z,x_{k_n}))_{n\in\NN}$ is bounded.
Therefore, \cite[Theorem~3.8(ii)]{BB97} and the essential
smoothness of $f$ yield $x\in\intdom f$.
\end{proof}

\begin{definition}
\label{d:f}
Algorithm~\ref{a:1} is focusing if, for every $z\in\mathscr{S}$, 
\begin{equation}
\label{e:f}
\begin{cases}
\big(D_{f_n}(z,x_n)\big)_{n\in\NN}\;\text{converges}\\
\displaystyle
\sum_{n\in\NN}\Pair{x_{n+1}-z}{\gamma_n^{-1}\big(\nabla f_n(x_n)
-\nabla f_n(x_{n+1})\big)-Bx_n+{B}z}<\pinf\\
\displaystyle
\sum_{n\in\NN}(1-\delta_2)\Pair{x_n-z}{Bx_n-Bz}<\pinf\\
\displaystyle
\sum_{n\in\NN}(1-\kappa\gamma_n/\alpha)D_{f_n}(x_{n+1},x_n)<\pinf
\end{cases}
\quad\Rightarrow\quad
\WC(x_n)_{n\in\NN}\subset\zer(A+B).
\end{equation}
\end{definition}

Our main result establishes the weak convergence of the
orbits of Algorithm~\ref{a:1}.

\begin{theorem}
\label{t:1}
Let $(x_n)_{n\in\NN}$ be a sequence generated by 
Algorithm~\ref{a:1} and suppose that the following hold:
\begin{enumerate}[label={\rm[\alph*]}]
\item
\label{t:1a}
$(x_n)_{n\in\NN}$ is bounded.
\item
\label{t:1b}
$\WC(x_n)_{n\in\NN}\subset\intdom f$.
\item
\label{t:1d}
Algorithm~\ref{a:1} is focusing. 
\item
\label{t:1c}
One of the following is satisfied:
\begin{enumerate}[label={\rm\arabic*/}]
\item
\label{t:1c1}
$\mathscr{S}$ is a singleton.
\item
\label{t:1c2}
There exists a function $g$ in $\Gamma_0(\XX)$ which is 
G\^ateaux differentiable on $\intdom g\supset C$, with
$\nabla g$ strictly monotone on $C$, and such that,
for every sequence $(y_n)_{n\in\NN}$ in $C$ and every
$y\in\WC(y_n)_{n\in\NN}\cap C$, 
$y_{k_n}\weakly y$ $\Rightarrow$
$\nabla f_{k_n}(y_{k_n})\weakly\nabla g(y)$.
\end{enumerate}
\end{enumerate}
Then $(x_n)_{n\in\NN}$ converges weakly to a point in
$\mathscr{S}$.
\end{theorem}
\begin{proof}
It results from \ref{t:1a} and the reflexivity of $\XX$ that
\begin{equation}
\label{e:2918}
(x_n)_{n\in\NN}\;\text{lies in a weakly sequentially compact set}.
\end{equation}
On the other hand, \ref{t:1d} and items
\ref{p:1i}--\ref{p:1vi} in
Proposition~\ref{p:1} yield $\WC(x_n)_{n\in\NN}\subset\zer(A+B)$.
In turn, it results from \ref{t:1b} that
\begin{equation}
\label{e:6034}
\emp\neq\WC(x_n)_{n\in\NN}\subset\mathscr{S}\subset C.
\end{equation}
In view of \cite[Lemma~1.35]{Livre1} applied in $\XX^{\text{weak}}$,
it remains to show that $\WC(x_n)_{n\in\NN}$ is a singleton.
If \ref{t:1c}\ref{t:1c1} holds, this follows from \eqref{e:6034}.
Now suppose that \ref{t:1c}\ref{t:1c2} holds, 
and take $y_1$ and $y_2$ in $\WC(x_n)_{n\in\NN}$, say
$x_{k_n}\weakly y_1$ and $x_{l_n}\weakly y_2$. Then
$y_1\in\mathscr{S}$ and $y_2\in\mathscr{S}$
by virtue of \eqref{e:6034}, and we therefore
deduce from Proposition~\ref{p:1}\ref{p:1i} that
$(D_{f_n}(y_1,x_n))_{n\in\NN}$ and
$(D_{f_n}(y_2,x_n))_{n\in\NN}$ converge.
However, condition~\ref{a:1b} in Algorithm~\ref{a:1}
and \cite[Lemma~5.31]{Livre1} assert that
$(D_{f_n}(y_1,y_2))_{n\in\NN}$ converges.
Hence, appealing to \cite[Proposition~2.3(ii)]{Sico03},
it follows that
$(\pair{y_1-y_2}{\nabla f_n(x_n)-\nabla f_n(y_2)})_{n\in\NN}
=(D_{f_n}(y_2,x_n)+D_{f_n}(y_1,y_2)-D_{f_n}(y_1,x_n))_{n\in\NN}$
converges. Set $\ell=\lim\pair{y_1-y_2}{\nabla f_n(x_n)
-\nabla f_n(y_2)}$. Since $(x_n)_{n\in\NN}$ is a sequence in
$C$, we infer from \eqref{e:6034} and \ref{t:1c}\ref{t:1c2} that
$\ell\leftarrow\pair{y_1-y_2}{\nabla f_{l_n}(x_{l_n})
-\nabla f_{l_n}(y_2)}\rightarrow\pair{y_1-y_2}{\nabla g(y_2)
-\nabla g(y_2)}=0$, which yields $\ell=0$.
However, invoking \ref{t:1c}\ref{t:1c2}, we obtain
$\ell\leftarrow\pair{y_1-y_2}{\nabla f_{k_n}(x_{k_n})
-\nabla f_{k_n}(y_2)}\rightarrow\pair{y_1-y_2}{\nabla g(y_1)
-\nabla g(y_2)}$. It therefore follows that
$\pair{y_1-y_2}{\nabla g(y_1)-\nabla g(y_2)}=0$ and hence from 
the strict monotonicity of $\nabla g$ on $C$ that $y_1=y_2$.
\end{proof}

\begin{example}
\label{ex:56}
We provide an example with operating conditions that are not
captured by any of the methods described in
\eqref{e:5}--\eqref{e:8}. Let $p\in\left]1,\pinf\right[$, let
$(\chi_n)_{n\in\NN}$ be a sequence in $\left[1,\pinf\right[$ such
that $\chi_n\to 1$, and let $(\eta_n)_{n\in\NN}$ be a summable
sequence in $\RP$ such that $(\forall n\in\NN)$
$\chi_{n+1}\leq(1+\eta_n)\chi_n$. We denote by
$z=(\zeta_k)_{k\in\NN}$
a sequence in $\ell^p(\NN)$. Set $\XX=\ell^p(\NN)\times\RR$,
hence $\XX^*=\ell^{p/(p-1)}(\NN)\times\RR$, and define the
Legendre functions
\begin{equation}
\label{e:l}
(\forall n\in\NN)\quad f_n\colon\XX\to\RX\colon(z,\xi)\mapsto
\begin{cases}
\dfrac{\chi_n}{p}\|z\|^p+1-\xi+\xi\ln\xi,
&\text{if}\;\;\xi>0;\\[4mm]
\dfrac{\chi_n}{p}\|z\|^p+1,&\text{if}\;\;\xi=0;\\
\pinf,&\text{if}\;\;\xi\leq 0
\end{cases}
\end{equation}
and
\begin{equation}
\label{e:l2}
f=g\colon\XX\to\RX\colon(z,\xi)\mapsto
\begin{cases}
\dfrac{1}{p}\|z\|^p-\xi+\xi\ln\xi,
&\text{if}\;\;\xi>0;\\[4mm]
\dfrac{1}{p}\|z\|^p,&\text{if}\;\;\xi=0;\\
\pinf,&\text{if}\;\;\xi\leq 0.
\end{cases}
\end{equation}
Now let $\psi\colon\XX\to\RP\colon(z,\xi)\mapsto\|z\|^p/p$, set
$B=\nabla\psi$, and let $A\colon\XX\to 2^{\XX^*}$ be any maximally
monotone operator such that
\begin{equation}
\dom A\subset\ell^p(\NN)\times\RPP\quad\text{and}\quad
\zer(A+B)\neq\emp. 
\end{equation}
Let us check that this setting conforms to that of
Theorem~\ref{t:1}.
First, Proposition~\ref{p:7}\ref{p:7v} implies that
\eqref{e:1d} is satisfied with $\delta_1=0$ and
$\delta_2=\kappa=1$. Next, we note that 
$\intdom f=\ell^p(\NN)\times\RPP$, that $(f_n)_{n\in\NN}$ lies in
$\BC_1(f)$, and that condition \ref{a:1b} in Algorithm~\ref{a:1}
holds. Furthermore, we derive from \eqref{e:l} that
\begin{equation}
\label{e:kf}
(\forall n\in\NN)\quad
\nabla f_n\colon\ell^p(\NN)\times\RPP\to\XX^*
\colon(z,\xi)\mapsto\Big(\chi_n\big(\operatorname{sign}
(\zeta_k)|\zeta_k|^{p-1}\big)_{k\in\NN},\ln\xi\Big)
\end{equation}
and we observe that
\begin{equation}
\label{e:kfu}
(\forall n\in\NN)\quad\ran\nabla f_n=\XX^*\quad
\text{and}\quad\dom(\gamma_n A)\subset\dom\nabla f_n.
\end{equation}
It therefore follows from the Br\'ezis--Haraux theorem
\cite[Th\'eor\`eme~4]{Brez76} that
\begin{equation}
\label{e:bh}
(\forall n\in\NN)\quad\ran(\nabla f_n+\gamma_n A)=\XX^*,
\end{equation}
and hence that condition \ref{a:1c} in Algorithm~\ref{a:1}
holds. It remains to verify condition~\ref{t:1c}\ref{t:1c2} in
Theorem~\ref{t:1}. 
Set $\varphi\colon\ell^p(\NN)\to\RP\colon z\mapsto\|z\|^p/p$ and
$(\forall n\in\NN)$ $\varphi_n\colon\ell^p(\NN)\to\RP\colon 
z\mapsto\chi_n\|z\|^p/p$.
Take a sequence
$(z_n,\xi_n)_{n\in\NN}$ in $\dom A$ and a point
$(z,\xi)\in\dom A$ such that $(z_n,\xi_n)\weakly(z,\xi)$.
We have $\xi_n\to\xi$ and $(\forall k\in\NN)$
$\zeta_{n,k}\to\zeta_k$. Now let
$(e_k)_{k\in\NN}$ be the canonical Schauder basis of $\ell^p(\NN)$.
Then
\begin{equation}
(\forall k\in\NN)\quad 
\Pair{e_k}{\nabla\varphi_n(z_n)}=
\chi_n\operatorname{sign}(\zeta_{n,k})|\zeta_{n,k}|^{p-1}\to
\operatorname{sign}(\zeta_{k})|\zeta_{k}|^{p-1}=
\Pair{e_k}{\nabla\varphi(z)}
\end{equation}
and $(\nabla\varphi_n(z_n))_{n\in\NN}$ is bounded. It therefore
follows from \cite[Th\'eor\`eme~VIII-2]{Bana32} that 
$\nabla\varphi_n(z_n)\weakly\nabla\varphi(z)$ and, in turn,
that $\nabla f_n(z_n,\xi_n)\weakly\nabla g(z,\xi)$ by
\eqref{e:l2} and \eqref{e:kf}. 
Note that the above setting is not covered by the assumptions
underlying \eqref{e:5}--\eqref{e:8}:
the fact that $B\neq 0$ excludes \cite{Sico03},
the fact that $\XX$ is not a Hilbert space excludes
\cite{Opti14} and \cite{Rena97},
and \cite{Nguy17} is excluded because
$A$ is not a subdifferential.
\end{example}

\section{Special cases and applications}
\label{sec:3}

We illustrate the general scope of Theorem~\ref{t:1} by 
recovering apparently unrelated results and also by deriving 
new ones. Sufficient conditions for 
\ref{t:1a} and~\ref{t:1b} in Theorem~\ref{t:1} to hold can be
found in Propositions~\ref{p:1}\ref{p:1vii} and \ref{p:2},
respectively. As to checking the focusing condition~\ref{t:1d}, the
following fact will be useful.

\begin{lemma}
\label{l:1}
{\rm\cite[Proposition~2.1(iii)]{Joca16}}
Let $M_1\colon\XX\to 2^{\XX^*}$ and $M_2\colon\XX\to 2^{\XX^*}$ be
maximally monotone, let $(a_n,a_n^*)_{n\in\NN}$ be a sequence 
in $\gra M_1$, let $(b_n,b_n^*)_{n\in\NN}$ be a sequence in 
$\gra M_2$, let $x\in\XX$, and let $y^*\in\XX^*$. 
Suppose that $a_n\weakly x$, $b_n^*\weakly y^*$, 
$a_n^*+b_n^*\to 0$, and $a_n-b_n\to 0$. Then 
$x\in\zer(M_1+M_2)$.
\end{lemma}

\subsection{Recovering existing frameworks for monotone inclusions}
\label{sec:31}

In this section, we show that the existing results of
\cite{Sico03,Opti14,Rena97} discussed in the Introduction can be
recovered from Theorem~\ref{t:1}. As will be clear from the proofs,
more general versions of these results can also be derived at once
from Theorem~\ref{t:1}.
First, we derive from Theorem~\ref{t:1} the convergence of the
Bregman-based proximal point algorithm \eqref{e:5} studied in
\cite[Section~5.5]{Sico03}.

\begin{corollary}
\label{c:1}
Let $A\colon\XX\to 2^{\XX^*}$ be maximally monotone, let 
$f\in\Gamma_0(\XX)$ be a supercoercive Legendre function such that
$\emp\neq\zer A\subset\dom A\subset\IDD$ and $\nabla f$ 
is weakly sequentially continuous, and let 
$(\gamma_n)_{n\in\NN}$ be a
sequence in $\RPP$ such that $\inf_{n\in\NN}\gamma_n>0$.
Suppose that, for every bounded sequence $(y_n)_{n\in\NN}$
in $\IDD$,
\begin{equation}
\label{e:1811}
D_f(y_{n+1},y_n)\to 0\quad\Rightarrow\quad
\nabla f(y_{n+1})-\nabla f(y_n)\to 0.
\end{equation}
Take $x_0\in C$ and set $(\forall n\in\NN)$
$x_{n+1}=(\nabla f+\gamma_nA\big)^{-1}(\nabla f(x_n))$.
Then $(x_n)_{n\in\NN}$ converges weakly to a point in $\zer A$.
\end{corollary}
\begin{proof}
We apply Theorem~\ref{t:1} with $B=0$, $\alpha=1$,
$\kappa=\delta_1=\delta_2=0$, and $(\forall n\in\NN)$ $f_n=f$.
First, \eqref{e:1d} together with
conditions~\ref{a:1a} and \ref{a:1b} in Algorithm~\ref{a:1}
are trivially fulfilled. On the other hand,
since $f$ is a Legendre function and $\dom A\subset\intdom f$,
condition~\ref{a:1c} in Algorithm~\ref{a:1}
follows from \cite[Theorem~3.13(iv)(d)]{Sico03}.
Next, condition~\ref{t:1a} in Theorem~\ref{t:1} follows from
Proposition~\ref{p:1}\ref{p:1vii}\ref{p:1viib}. Furthermore,
in view of the weak sequential continuity of $\nabla f$,
condition~\ref{t:1c}\ref{t:1c2} in Theorem~\ref{t:1}
is satisfied with $g=f$. Next, to show that the algorithm
is focusing, suppose that
$\sum_{n\in\NN}D_f(x_{n+1},x_n)<\pinf$
and take $x\in\WC(x_n)_{n\in\NN}$, say $x_{k_n}\weakly x$.
Since $(x_n)_{n\in\NN}$ is a bounded sequence in $\intdom f$,
we derive from \eqref{e:1811} that
$\nabla f(x_{n+1})-\nabla f(x_n)\to 0$.
In turn, since $\inf_{n\in\NN}\gamma_n>0$,
it follows that $\gamma_n^{-1}(\nabla f(x_{n+1})-\nabla f(x_n))
\to 0$. However, by construction,
$(\forall n\in\NN)$ $\gamma_{k_n-1}^{-1}
(\nabla f(x_{k_n-1})-\nabla f(x_{k_n}))\in
Ax_{k_n}$. Therefore, upon invoking Lemma~\ref{l:1} (with
$M_1=A$ and $M_2=0$),
we obtain $x\in\zer A$ and the algorithm is therefore focusing.
This also shows that $\WC(x_n)_{n\in\NN}\subset\zer
A\subset\intdom f$. Condition~\ref{t:1b} in Theorem~\ref{t:1}
is thus satisfied.
\end{proof}

The next application of Theorem~\ref{t:1} is a variable metric
version of the Hilbertian forward-backward method 
\eqref{e:4} established in \cite[Theorem~4.1]{Opti14}.

\begin{corollary}
\label{c:2}
Let $\XX$ be a real Hilbert space, let $A\colon\XX\to 2^\XX$
be maximally monotone, let $\alpha$ and $\beta$ be in $\RPP$,
and let $B\colon\XX\to\XX$ satisfy
\begin{equation}
\label{e:6512}
(\forall x\in\XX)(\forall y\in\XX)\quad
\scal{x-y}{Bx-By}\geq\beta\|Bx-By\|^2.
\end{equation}
Further, for every $n\in\NN$, let $U_n\colon\XX\to\XX$ be
a bounded linear operator which is $\alpha$-strongly monotone
and self-adjoint. Suppose that $\zer(A+B)\neq\emp$ and that 
there exists a summable sequence $(\eta_n)_{n\in\NN}$ in $\RP$
such that
\begin{equation}
\label{e:9854}
(\forall n\in\NN)(\forall x\in\XX)\quad
\scal{x}{U_{n+1}x}\leq(1+\eta_n)\scal{x}{U_nx}.
\end{equation}
Let $\varepsilon\in\left]0,2\beta\right[$
and let $(\gamma_n)_{n\in\NN}$ be a sequence in $\RPP$
such that $0<\inf_{n\in\NN}\gamma_n\leq\sup_{n\in\NN}\gamma_n
\leq(2\beta-\varepsilon)\alpha$. Define a sequence
$(x_n)_{n\in\NN}$ via the recursion
\begin{equation}
\label{e:8969}
x_0\in\dom A\quad\text{and}\quad
(\forall n\in\NN)\quad
x_{n+1}=(U_n+\gamma_nA)^{-1}(U_nx_n-\gamma_nBx_n).
\end{equation}
Then $(x_n)_{n\in\NN}$ converges weakly to a point in
$\zer(A+B)$.
\end{corollary}
\begin{proof}
Set $f=\|\mute\|^2/2$, $C=\dom A$, and
$\mathscr{S}=\zer(A+B)$. In addition, for every $n\in\NN$,
define $f_n\colon\XX\to\RR\colon x\mapsto\scal{x}{U_nx}/2$.
Let us apply Theorem~\ref{t:1} with
$\kappa=1/(2\beta-\varepsilon)$, $\delta_1=0$, and
$\delta_2=(2\beta-\varepsilon)/(2\beta)\in\zeroun$.
First, $f\in\Gamma_0(\XX)$ is a
supercoercive Legendre function with
$\dom f=\XX$ and, for every $n\in\NN$,
since $\nabla f_n=U_n$ is $\alpha$-strongly monotone,
$f_n\in\BC_\alpha(f)$.
Furthermore, it follows from
Proposition~\ref{p:7}\ref{p:7ii}\ref{p:7iia}
that \eqref{e:1d} is fulfilled. 
We also observe that condition~\ref{a:1a}
in Algorithm~\ref{a:1} is satisfied.
Next, by \eqref{e:9854} and the assumption that the operators
$(U_n)_{n\in\NN}$ are self-adjoint,
\begin{align}
(\forall n\in\NN)(\forall x\in\XX)(\forall y\in\XX)\quad
D_{f_{n+1}}(x,y)
&=\frac{1}{2}\scal{x-y}{U_{n+1}(x-y)} \nonumber\\
&\leq\frac{1+\eta_n}{2}\scal{x-y}{U_n(x-y)}
\nonumber\\
&=D_{f_n}(x,y)
\end{align}
and condition~\ref{a:1b} in Algorithm~\ref{a:1} therefore
holds. Now take $n\in\NN$. Since $\nabla f_n=U_n$ is
maximally monotone with $\dom\nabla f_n=\XX$ and $A$ is maximally
monotone, \cite[Corollary~25.5(i)]{Livre1} entails that 
$\nabla f_n+\gamma_nA$ is maximally monotone.
Thus, since $\nabla f_n+\gamma_n A$ is $\alpha$-strongly
monotone, \cite[Proposition~22.11(ii)]{Livre1} implies that
$\ran(\nabla f_n+\gamma_n A)=\XX$ and it follows that
condition~\ref{a:1c} in Algorithm~\ref{a:1} is satisfied.
Next, in view of
Proposition~\ref{p:1}\ref{p:1vii}\ref{p:1viib}, $(x_n)_{n\in\NN}$
is bounded, while $\WC(x_n)_{n\in\NN}\subset\XX=\intdom f$. 
Now set $\mu=\sup_{n\in\NN}\|U_n\|$.
For every $n\in\NN$, since it results from
\eqref{e:9854} and \cite[Fact~2.25(iii)]{Livre1} that
\begin{equation}
(\forall x\in\XX)\quad
\|x\|\leq 1\quad\Rightarrow\quad
\scal{x}{U_nx}\leq\bigg(\prod_{k\in\NN}(1+\eta_k)\bigg)\scal{x}{U_0x}
\leq\bigg(\prod_{k\in\NN}(1+\eta_k)\bigg)\|U_0\|,
\end{equation}
we derive from \cite[Fact~2.25(iii)]{Livre1} that
$\|U_n\|\leq\|U_0\|\prod_{k\in\NN}(1+\eta_k)$. Hence $\mu<\pinf$
and therefore, appealing to \cite[Lemma~2.3(i)]{Guad2012}, there
exists an $\alpha$-strongly monotone self-adjoint bounded linear
operator $U\colon\XX\to\XX$ such that
$(\forall w\in\XX)$ $U_nw\to Uw$. Define
$g\colon\XX\to\XX:x\mapsto\scal{x}{Ux}/2$. Then $\nabla g=U$
is strongly monotone (and thus strictly monotone). Furthermore,
given $(y_n)_{n\in\NN}$ in $C$ and $y\in\WC(y_n)_{n\in\NN}\cap C$,
say $y_{k_n}\weakly y$, we have
\begin{equation}
(\forall w\in\XX)\quad
\scal{w}{\nabla f_{k_n}(y_{k_n})}=\scal{U_{k_n}w}{y_{k_n}}
\to\scal{Uw}{y}=\scal{w}{Uy}=\scal{w}{\nabla g(y)}
\end{equation}
and thus $\nabla f_{k_n}(y_{k_n})\weakly\nabla g(y)$.
Therefore, condition~\ref{t:1c}\ref{t:1c2} in Theorem~\ref{t:1} is
satisfied. Let us now verify that \eqref{e:8969} is
focusing. Towards this goal,
take $z\in\mathscr{S}$ and
suppose that $\sum_{n\in\NN}(1-\delta_2)\scal{x_n-z}{Bx_n-Bz}
<\pinf$ and $\sum_{n\in\NN}(1-\kappa\gamma_n/\alpha)
D_{f_n}(x_{n+1},x_n)<\pinf$.
Since $\delta_2<1$ and $\sup_{n\in\NN}(\kappa\gamma_n)<\alpha$,
we infer from \eqref{e:6512} that
\begin{equation}
\label{e:3026}
\sum_{n\in\NN}\|Bx_{n}-Bz\|^2
\leq\frac{1}{\beta}\sum_{n\in\NN}\scal{x_n-z}{Bx_n-Bz}<\pinf
\end{equation}
and $\sum_{n\in\NN}\|x_{n+1}-x_n\|^2
=2\sum_{n\in\NN}D_f(x_{n+1},x_n)
\leq(2/\alpha)\sum_{n\in\NN}D_{f_n}(x_{n+1},x_n)<\pinf$.
It follows that
\begin{equation}
\label{e:5029}
\|U_n(x_{n+1}-x_n)\|\leq\mu\|x_{n+1}-x_n\|\to 0.
\end{equation}
Now take $x\in\WC(x_n)_{n\in\NN}$, say $x_{k_n}\weakly x$,
and set $(\forall n\in\NN)$
$x_{n+1}^*=\gamma_n^{-1}U_n(x_n-x_{n+1})-Bx_n$.
It results from \eqref{e:8969} that
$(x_{k_n+1},x_{k_n+1}^*)_{n\in\NN}$ lies in $\gra A$
and from \eqref{e:5029} that $x_{k_n+1}\weakly x$.
Moreover, \eqref{e:5029} yields $x_{k_n+1}^*+Bx_{k_n}\to 0$.
Altogether, Lemma~\ref{l:1}
(applied to the sequences $(x_{k_n+1},x_{k_n+1}^*)_{n\in\NN}$
in $\gra A$ and $(x_{k_n},Bx_{k_n})_{n\in\NN}$ in $\gra B$)
guarantees that $x\in\zer(A+B)$.
Consequently, Theorem~\ref{t:1} asserts that $(x_n)_{n\in\NN}$
converges weakly to a point in $\mathscr{S}$.
\end{proof}

\begin{example}
\label{c:0}
The classical forward-backward method is obtained
by setting $U_n\equiv\Id$ in Corollary~\ref{c:2}, which yields
\begin{equation}
\label{e:paul}
x_0\in\dom A\quad\text{and}\quad(\forall n\in\NN)\quad
x_{n+1}=(\Id+\gamma_nA)^{-1}(x_n-\gamma_nBx_n).
\end{equation}
The case when the proximal parameters $(\gamma_n)_{n\in\NN}$
are constant was first addressed in \cite{Merc79}.
\end{example}

We now turn to the Renaud--Cohen algorithm \eqref{e:8}
and recover \cite[Theorem~3.4]{Rena97}.

\begin{corollary}
Let $\XX$ be a real Hilbert space, let $A\colon\XX\to 2^\XX$
and $B\colon\XX\to\XX$ be maximally monotone,
and let $f\colon\XX\to\RR$ be convex and Fr\'echet differentiable.
Suppose that $\zer(A+B)\neq\emp$, that
$\nabla f$ is $1$-strongly monotone on $\dom A$
and Lipschitzian on bounded sets, and that there exists
$\beta\in\RPP$ such that
\begin{equation}
\label{e:1917}
\big(\forall(x,x^*)\in\gra(A+B)\big)\big(\forall(y,y^*)
\in\gra(A+B)\big)\quad\scal{x-y}{x^*-y^*}\geq\beta\|Bx-By\|^2.
\end{equation}
Let $\gamma\in\left]0,2\beta\right[$,
take $x_0\in\dom A$, and set $(\forall n\in\NN)$
$x_{n+1}=(\nabla f+\gamma A)^{-1}(\nabla f(x_n)-\gamma Bx_n)$.
Suppose, in addition, that $\nabla f$ is
weakly sequentially continuous. 
Then $(x_n)_{n\in\NN}$ converges weakly to a point in
$\zer(A+B)$.
\end{corollary}
\begin{proof}
Let $\varepsilon\in\left]0,2\beta\right[$ be such that
$\gamma<2\beta-\varepsilon$.
We apply Theorem~\ref{t:1} with $C=\dom A$, $\alpha=1$,
$\kappa=1/(2\beta-\varepsilon)$, $\delta_1=\delta_2=
(2\beta-\varepsilon)/(2\beta)\in\zeroun$, and 
$(\forall n\in\NN)$ $f_n=f$ and $\eta_n=0$.
Proposition~\ref{p:7}\ref{p:7iii} asserts that
\eqref{e:1d} is satisfied.
Furthermore, as shown in the proof of
Proposition~\ref{p:7}\ref{p:7iii},
\begin{equation}
\label{e:4412}
(\forall x\in\cdom A)(\forall y\in\cdom A)\quad
D_f(x,y)\geq\frac{1}{2}\|x-y\|^2.
\end{equation}
Next, note that conditions~\ref{a:1a} and \ref{a:1b} 
in Algorithm~\ref{a:1} are trivially
satisfied. Since $\nabla f+\gamma A$ is strongly
monotone and since, by \cite[Corollary~25.5(i)]{Livre1},
$\nabla f+\gamma A$ is maximally monotone,
it follows from \cite[Proposition~22.11(ii)]{Livre1} that
$\ran(\nabla f+\gamma A)=\XX$ and therefore that
condition~\ref{a:1c} in Algorithm~\ref{a:1} holds.
We observe that
condition~\ref{t:1b} in Theorem~\ref{t:1} is trivially satisfied
and that condition~\ref{t:1a} in Theorem~\ref{t:1} follows from
\eqref{e:4412} and Proposition~\ref{p:1}\ref{p:1i}.
Furthermore, since $\nabla f$ is weakly sequentially
continuous and $1$-strongly monotone on $C$,
condition~\ref{t:1c}\ref{t:1c2} in Theorem~\ref{t:1} is satisfied
with $g=f$. Now take $z\in\zer(A+B)$
and suppose that
$\sum_{n\in\NN}(1-\kappa\gamma)D_f(x_{n+1},x_n)<\pinf$,
$\sum_{n\in\NN}(1-\delta_2)\scal{x_n-z}{Bx_n-Bz}<\pinf$,
and $\sum_{n\in\NN}\scal{x_{n+1}-z}{\gamma^{-1}(\nabla
f(x_n)-\nabla f(x_{n+1}))-Bx_n+Bz}<\pinf$. Then, since
$\kappa\gamma<1$ and $\delta_2<1$, it follows that
\begin{equation}
\label{e:1563}
\displaystyle
\Sum_{n\in\NN}D_f(x_{n+1},x_n)<\pinf
\quad\text{and}\quad
\Sum_{n\in\NN}\scal{x_n-z}{Bx_n-Bz}<\pinf,
\end{equation}
and therefore that
\begin{equation}
\label{e:1564}
\displaystyle
\sum_{n\in\NN}\sscal{x_{n+1}-z}{\gamma^{-1}(\nabla
f(x_n)-\nabla f(x_{n+1}))-Bx_n+Bx_{n+1}}<\pinf.
\end{equation}
Since $(z,0)\in\gra(A+B)$ and since the sequence
$(x_{n+1},\gamma^{-1}(\nabla f(x_n)-\nabla f(x_{n+1}))
-Bx_n+Bx_{n+1})_{n\in\NN}$ lies in $\gra(A+B)$ by construction,
it follows from \eqref{e:1917} and \eqref{e:1564} that 
$\sum_{n\in\NN}\|Bx_n-Bz\|^2<\pinf$.
On the other hand, since $(x_n)_{n\in\NN}$ lies in $\dom A$ by
Proposition~\ref{p:1}, we deduce from \eqref{e:4412}
and \eqref{e:1563} that
$x_{n+1}-x_n\to 0$. In turn, it results from the Lipschitz
continuity of $\nabla f$ on the bounded set
$\{x_n\}_{n\in\NN}$ that $\nabla f(x_n)-\nabla f(x_{n+1})\to
0$. Now take $x\in\WC(x_n)_{n\in\NN}$, say $x_{k_n}\weakly x$,
and set $(\forall n\in\NN)$ $x_{n+1}^*
=\gamma^{-1}(\nabla f(x_n)-\nabla f(x_{n+1}))-Bx_n$.
Then $(x_{k_n+1},x_{k_n+1}^*)_{n\in\NN}$ lies in
$\gra A$. Furthermore, $x_{k_n+1}^*+Bx_{k_n}
=\gamma^{-1}(\nabla f(x_{k_n})-\nabla f(x_{k_n+1}))\to 0$
and, since $x_n-x_{n+1}\to 0$, $x_{k_n+1}\weakly x$.
Thus, applying Lemma~\ref{l:1} with the sequences
$(x_{k_n+1},x_{k_n+1}^*)_{n\in\NN}$ and
$(x_{k_n},Bx_{k_n})_{n\in\NN}$ yields $x\in\zer(A+B)$, and we
conclude that condition~\ref{t:1d} in Theorem~\ref{t:1} is
satisfied as well.
\end{proof}

\subsection{The finite-dimensional case}
\label{sec:32}

We discuss the finite-dimensional case, a setting in which the 
assumptions can be greatly simplified and the results presented
below are new. 

\begin{corollary}
\label{c:wu}
Let $(x_n)_{n\in\NN}$ be a sequence generated by
Algorithm~\ref{a:1}. In addition, suppose that the following
hold:
\begin{enumerate}[label={\rm[\alph*]}]
\item
\label{p:10a}
$\XX$ is finite-dimensional.
\item
\label{p:10b}
$f$ is essentially strictly convex and $\dom f^*$ is open.
\item
\label{p:10c}
$(\intdom f)\cap\cdom A\subset\intdom B$.
\item
\label{p:10d}
$\sup_{n\in\NN}(\kappa\gamma_n)<\alpha$.
\item
\label{p:10e}
There exists a function $g$ in $\Gamma_0(\XX)$ which is 
differentiable on $\intdom g\supset\intdom f$, with
$\nabla g$ strictly monotone on $C$, and such that,
for every sequence $(y_n)_{n\in\NN}$ in $C$ and every
sequential cluster point $y\in\intdom f$ of $(y_n)_{n\in\NN}$, 
$y_{k_n}\to y$ $\Rightarrow$
$\nabla f_{k_n}(y_{k_n})\to\nabla g(y)$.
\end{enumerate}
Then $(x_n)_{n\in\NN}$ converges to a point in $\mathscr{S}$.
\end{corollary}
\begin{proof}
It follows from
Proposition~\ref{p:1}\ref{p:1vii}\ref{p:1viid} that
$(x_n)_{n\in\NN}$ is bounded and from
Proposition~\ref{p:2}\ref{p:2d} that
$\WC(x_n)_{n\in\NN}\subset\intdom f$.
In view of Theorem~\ref{t:1}, it remains to show that
Algorithm~\ref{a:1} is focusing.
Towards this goal, let $z\in\mathscr{S}$,
and suppose that $(D_{f_n}(z,x_n))_{n\in\NN}$
converges and $\sum_{n\in\NN}
(1-\kappa\gamma_n/\alpha)D_{f_n}(x_{n+1},x_n)<\pinf$,
and let $x$ be a sequential cluster point of 
$(x_n)_{n\in\NN}$, say $x_{k_n}\to x$.
Using \ref{p:10d} and the fact that
$(f_n)_{n\in\NN}$ lies in $\BC_\alpha(f)$, we obtain
\begin{equation}
\label{e:3037}
\big(D_f(z,x_n)\big)_{n\in\NN}\;\text{is bounded}
\quad\text{and}\quad
\sum_{n\in\NN}D_{f_n}(x_{n+1},x_n)<\pinf.
\end{equation}
Since $(x_{k_n})_{n\in\NN}$ lies in $\intdom f$,
\cite[Theorem~3.8(ii)]{BB97} and \eqref{e:3037} imply that
\begin{equation}
\label{e:8440}
x\in\intdom f
\end{equation}
and \cite[Theorem~5.10]{Ccm01} thus yields
\begin{equation}
\label{e:2697}
\nabla f(x_{k_n})\to\nabla f(x)\in\intdom f^*.
\end{equation}
Next, it results from \ref{p:10b},
\cite[Lemma~7.3(vii)]{Ccm01}, and \eqref{e:3037} that
\begin{equation}
\label{e:0285}
\big(D_{f^*}(\nabla f(x_n),\nabla f(z))\big)_{n\in\NN}
=\big(D_f(z,x_n)\big)_{n\in\NN}\;\text{is bounded}.
\end{equation}
Therefore, since $\nabla f(z)\in\intdom f^*$ 
\cite[Theorem~5.10]{Ccm01} and since $f^*$ is a Legendre
function \cite[Corollary~5.5]{Ccm01}, it results from
\cite[Lemma~7.3(v)]{Ccm01} that $(\nabla f(x_{k_n+1}))_{n\in\NN}$ 
is bounded. In turn, there exists a
strictly increasing sequence $(l_{k_n})_{n\in\NN}$ in
$\NN$ and a point $x^*\in\XX^*$ such that
\begin{equation}
\label{e:8291}
\nabla f(x_{l_{k_n}+1})\to x^*.
\end{equation}
By lower semicontinuity of $D_{f^*}(\mute,\nabla
f(z))$ and \eqref{e:0285}, $x^*\in\dom f^*$. On the other
hand, appealing to \cite[Lemma~7.3(vii)]{Ccm01} and
\eqref{e:3037}, we obtain
\begin{equation}
0\leq D_{f^*}\big(\nabla f(x_{l_{k_n}}),
\nabla f(x_{l_{k_n}+1})\big)
=D_f\big(x_{l_{k_n}+1},x_{l_{k_n}}\big)
\leq\frac{1}{\alpha}D_{f_{l_{k_n}}}
\big(x_{l_{k_n}+1},x_{l_{k_n}}\big)
\to 0.
\end{equation}
Thus, since $(\nabla f(x_n))_{n\in\NN}$ lies in $\intdom f^*$
by virtue of Proposition~\ref{p:1} and
\cite[Theorem~5.10]{Ccm01}, we derive from
\cite[Theorem~3.9(iii)]{BB97}, \eqref{e:2697}, and \eqref{e:8291}
that $x^*=\nabla f(x)$ and, hence, from \eqref{e:8291} that
$\nabla f(x_{l_{k_n}+1})\to\nabla f(x)$.
It thus follows from \cite[Theorem~5.10]{Ccm01} that
$x_{l_{k_n}+1}\to x$. In turn, by using respectively \ref{p:10e}
with the sequences $(x_n)_{n\in\NN}$ and $(x_{n+1})_{n\in\NN}$,
we get $\nabla f_{l_{k_n}}(x_{l_{k_n}})\to\nabla g(x)$ and
$\nabla f_{l_{k_n}}(x_{l_{k_n}+1})\to\nabla g(x)$.
Now set $(\forall n\in\NN)$ $x_{n+1}^*=\gamma_n^{-1}
(\nabla f_n(x_n)-\nabla f_n(x_{n+1}))-Bx_n$.
Then, by construction of $(x_n)_{n\in\NN}$,
$(\forall n\in\NN)$ $(x_{n+1},x_{n+1}^*)\in\gra A$.
In addition, since $\inf_{n\in\NN}\gamma_n>0$
and $\nabla f_{l_{k_n}}(x_{l_{k_n}})
-\nabla f_{l_{k_n}}(x_{l_{k_n}+1})\to\nabla g(x)-\nabla g(x)=0$, 
we deduce that $x_{l_{k_n}+1}^*+Bx_{l_{k_n}}\to 0$.
On the other hand, since $(x_n)_{n\in\NN}$ lies in $\dom A$
and $x_{k_n}\to x$, it follows that $x\in\cdom A$
and therefore, by \eqref{e:8440} and \ref{p:10c}, that
$x\in\intdom B$. Hence, using \cite[Corollary~1.1]{Rock69},
we obtain $Bx_{l_{k_n}}\to Bx$.
Altogether, Lemma~\ref{l:1} (applied to the sequence
$(x_{l_{k_n}+1},x_{l_{k_n}+1}^*)_{n\in\NN}$ in $\gra A$
and the sequence $(x_{l_{k_n}},Bx_{l_{k_n}})_{n\in\NN}$ in
$\gra B$) asserts that $x\in\zer(A+B)$. 
In view of Theorem~\ref{t:1}, we conclude that $(x_n)_{n\in\NN}$
converges to a point in $\mathscr{S}$.
\end{proof}

\subsection{Forward-backward splitting for convex minimization}
\label{sec:mini}

In this section, we study the convergence of \eqref{e:7}. Our
results improve on and complement those of \cite{Nguy17}. 

\begin{problem}
\label{prob:2}
Let $\varphi\in\Gamma_0(\XX)$, let $\psi\in\Gamma_0(\XX)$,
and let $f\in\Gamma_0(\XX)$ be essentially smooth.
Set $C=(\intdom f)\cap\dom\partial\varphi$ 
and $\mathscr{S}=(\intdom f)\cap\Argmin(\varphi+\psi)$.
Suppose that $\varphi+\psi$ is coercive,
$\emp\neq C\subset\intdom\psi$, $\mathscr{S}\neq\emp$, 
$\psi$ is G\^ateaux differentiable on $\intdom\psi$, and
there exists $\kappa\in\RPP$ such that
\begin{equation}
\label{e:2d}
(\forall x\in C)(\forall y\in C)\quad
D_{\psi}(x,y)\leq\kappa D_{f}(x,y).
\end{equation}
The objective is to find a point in $\mathscr{S}$.
\end{problem}

In the context of Problem~\ref{prob:2}, given 
$\gamma\in\RPP$ and $g\in\BC_{\alpha}(f)$, we define
$\prox^{g}_{\gamma\varphi}=(\nabla g+\gamma\partial\varphi)^{-1}$.

\begin{algorithm}
\label{a:2}
Consider the setting of Problem~\ref{prob:2}.
Let $\alpha\in\RPP$,
let $(\gamma_n)_{n\in\NN}$ be in $\RPP$, and
let $(f_n)_{n\in\NN}$ be in $\BC_\alpha(f)$.
Suppose that the following hold:
\begin{enumerate}[label={\rm[\alph*]}]
\item
\label{a:2a}
There exists $\varepsilon\in\zeroun$
such that
$0<\inf_{n\in\NN}\gamma_n\leq\sup_{n\in\NN}\gamma_n
\leq\alpha(1-\varepsilon)/\kappa$.
\item
\label{a:2b}
There exists a summable sequence $(\eta_n)_{n\in\NN}$ in $\RP$
such that $(\forall n\in\NN)$
$D_{f_{n+1}}\leq (1+\eta_n)D_{f_n}$.
\item
\label{a:2c}
For every $n\in\NN$, $\intdom f_n=\dom\partial f_n$
and $\nabla f_n$ is strictly monotone on $C$.
\end{enumerate}
Take $x_0\in C$ and set $(\forall n\in\NN)$
$x_{n+1}=\prox^{f_n}_{\gamma_n\varphi}(\nabla f_n(x_n)-\gamma_n
\nabla\psi(x_n))$.
\end{algorithm}

\begin{theorem}
\label{t:3}
Let $(x_n)_{n\in\NN}$ be a sequence generated by 
Algorithm~\ref{a:2} and suppose that the following hold:
\begin{enumerate}[label={\rm[\alph*]}]
\item
$\WC(x_n)_{n\in\NN}\subset\intdom f$.
\item
One of the following is satisfied:
\begin{enumerate}[label={\rm\arabic*/}]
\item
\label{t:3b1}
$\mathscr{S}$ is a singleton.
\item
\label{t:3b2}
There exists a function $g$ in $\Gamma_0(\XX)$ which is 
G\^ateaux differentiable on $\intdom g\supset C$, with
$\nabla g$ strictly monotone on $C$, and such that,
for every sequence $(y_n)_{n\in\NN}$ in $C$ and every
$y\in\WC(y_n)_{n\in\NN}\cap C$, 
$y_{k_n}\weakly y$ $\Rightarrow$
$\nabla f_{k_n}(y_{k_n})\weakly\nabla g(y)$.
\end{enumerate}
\end{enumerate}
Then the following hold:
\begin{enumerate}
\item
\label{t:3i}
$(x_n)_{n\in\NN}$ converges weakly to a point in
$\mathscr{S}$.
\item
\label{t:3ii-}
$(x_n)_{n\in\NN}$ is a monotone minimizing sequence:
$\varphi(x_n)+\psi(x_n)\downarrow\min(\varphi+\psi)(\XX)$.
\item
\label{t:3ii}
$\sum_{n\in\NN}((\varphi+\psi)(x_n)-\min(\varphi+\psi)(\XX))<\pinf$
and $(\varphi+\psi)(x_n)-\min(\varphi+\psi)(\XX)=o(1/n)$.
\item
\label{t:3iii}
$\sum_{n\in\NN}n
(D_{f_n}(x_{n+1},x_n)+D_{f_n}(x_n,x_{n+1}))<\pinf$.
\end{enumerate}
\end{theorem}
\begin{proof}
\ref{t:3i}:
We shall derive this result from Theorem~\ref{t:1} with
$A=\partial\varphi$, $B=\partial\psi$, $\delta_1=0$,
and $\delta_2=1$.
First, appealing to \cite[Theorem~2.4.4(i)]{Zali02},
$B$ is single-valued on $\intdom B=\intdom\psi$ and $B=\nabla\psi$
on $\intdom B$. Next, set $\theta=\varphi+\psi$.
Since $\emp\neq(\intdom f)\cap\dom\partial
\varphi\subset\intdom\psi$, we have
$\dom\varphi\cap\intdom\psi\neq\emp$. Hence,
\cite[Theorem~4.1.19]{Borw10} yields $A+B=\partial\theta$.
Therefore, $\Argmin\theta=\zer\partial\theta=\zer(A+B)$
and $\mathscr{S}=(\intdom f)\cap\zer(A+B)$.
Next, in view of Proposition~\ref{p:7}\ref{p:7v},
\eqref{e:1d} is fulfilled.
On the other hand, conditions~\ref{a:1a} and \ref{a:1b}
in Algorithm~\ref{a:1} are trivially satisfied.
To verify condition~\ref{a:1c} in Algorithm~\ref{a:1},
it suffices to show that,
for every $n\in\NN$, $(\nabla f_n-\gamma_nB)(C)
\subset\ran(\nabla f_n+\gamma_n A)$, i.e.,
since $C\subset\intdom B$ and $B=\nabla\psi$ on $\intdom B$,
that $(\nabla f_n-\gamma_n\nabla\psi)(C)
\subset\ran(\nabla f_n+\gamma_n A)$. To do so,
fix temporarily $n\in\NN$, let $x\in C$, and set
\begin{equation}
A_n=\nabla f_n+\gamma_nA-\nabla f_n(x)+\gamma_n\nabla\psi(x).
\end{equation}
Then, since $\dom\partial f_n\cap\dom A=(\intdom f_n)\cap\dom A
=(\intdom f)\cap\dom A\neq\emp$ by condition~\ref{a:2c} in
Algorithm~\ref{a:2}, it results from
\cite[Proposition~3.12]{Sico03} that $A_n$ is maximally
monotone. Next, we deduce from condition~\ref{a:2a} in
Algorithm~\ref{a:2} and \eqref{e:2d} that
\begin{equation}
\label{e:2604}
(\forall u\in C)(\forall v\in C)\quad
\gamma_nD_\psi(u,v)\leq\alpha(1-\varepsilon)D_\psi(u,v)/\kappa
\leq\alpha(1-\varepsilon)D_f(u,v)
\leq(1-\varepsilon)D_{f_n}(u,v).
\end{equation}
In turn,
\begin{align}
(\forall u\in C)(\forall v\in C)\quad
\gamma_n\pair{u-v}{\nabla\psi(u)-\nabla\psi(v)}
&=\gamma_n\big(D_\psi(u,v)+D_\psi(v,u)\big)
\nonumber\\
&\leq(1-\varepsilon)\big(D_f(u,v)+D_f(v,u)\big)
\nonumber\\
&=(1-\varepsilon)\pair{u-v}{\nabla f_n(u)-\nabla f_n(v)}.
\label{e:8523}
\end{align}
However, by coercivity of $\theta$, there exists
$\rho\in\RPP$ such that
\begin{equation}
\label{e:1557}
(\forall y\in\XX)\quad
\|y\|\geq\rho\quad\Rightarrow\quad
\inf\pair{y}{(A+B)(y+x)}=\inf\pair{y}{\partial\theta(y+x)}
\geq\theta(y+x)-\theta(x)\geq 0.
\end{equation}
Now suppose that $(y,y^*)\in\gra A_n(\mute+x)$ satisfies
$\|y\|\geq\rho$. Then $y+x\in\dom\nabla f_n\cap\dom A
=(\intdom f_n)\cap\dom A=C$ and
$y^*-\nabla f_n(y+x)+\gamma_n\nabla\psi(y+x)
+\nabla f_n(x) -\gamma_n\nabla\psi(x)\in\gamma_n(A+B)(y+x)$.
Thus, it follows from \eqref{e:1557} and \eqref{e:8523} that
\begin{equation}
0\leq\pair{y}{y^*}-\Pair{(y+x)-x}{(\nabla
f_n-\gamma_n\nabla\psi)(y+x)-(\nabla
f_n-\gamma_n\nabla\psi)(x)}
\leq\pair{y}{y^*}.
\end{equation}
Therefore, in view of \cite[Proposition~2]{Rock70} and
the maximal monotonicity of $A_n(\mute+x)$,
there exists $\overline{y}\in\XX$ such that
$0\in A_n(\overline{y}+x)$.
Hence $(\nabla f_n-\gamma_n\nabla\psi)(x)
\in\nabla f_n(\overline{y}+x)+\gamma_nA(\overline{y}+x)
\subset\ran(\nabla f_n+\gamma_n A)$, as desired.
Since $(x_{n+1},\gamma_n^{-1}(
\nabla f_n(x_n)-\nabla f_n(x_{n+1}))-\nabla\psi(x_n))$ 
lies in $\gra\partial\varphi$ by construction,
we derive from \cite[Proposition~2.3(ii)]{Sico03} that
\begin{align}
\label{e:9316}
(\forall x\in C)\quad
\varphi(x)
&\geq\varphi(x_{n+1})-\pair{x-x_{n+1}}{\nabla\psi(x_n)}
+\gamma_n^{-1}\pair{x-x_{n+1}}
{\nabla f_n(x_n)-\nabla f_n(x_{n+1})}
\nonumber\\
&\geq\varphi(x_{n+1})-\pair{x-x_{n+1}}{\nabla\psi(x_n)}
\nonumber\\
&\quad\;+\gamma_n^{-1}\big(D_{f_n}(x,x_{n+1})+D_{f_n}(x_{n+1},x_n)
-D_{f_n}(x,x_n)\big).
\end{align}
On the other hand, \eqref{e:2604} and the convexity of $\psi$
entail that
\begin{align}
(\forall x\in C)\quad
\psi(x_{n+1})
&\leq\psi(x_n)+\pair{x_{n+1}-x_n}{\nabla\psi(x_n)}
+(1-\varepsilon)\gamma_n^{-1}D_{f_n}(x_{n+1},x_n)
\nonumber\\
&=\psi(x_n)+\pair{x-x_n}{\nabla\psi(x_n)}
+\pair{x_{n+1}-x}{\nabla\psi(x_n)}
\nonumber\\
&\quad\;+(1-\varepsilon)\gamma_n^{-1}D_{f_n}(x_{n+1},x_n)
\nonumber\\
&\leq\psi(x)+\pair{x_{n+1}-x}{\nabla\psi(x_n)}
+(1-\varepsilon)\gamma_n^{-1}D_{f_n}(x_{n+1},x_n).
\label{e:6584}
\end{align}
Altogether, upon adding \eqref{e:9316} and \eqref{e:6584}, we
obtain
\begin{equation}
\label{e:7610}
(\forall x\in C)\quad
\theta(x_{n+1})+\gamma_n^{-1}D_{f_n}(x,x_{n+1})
+\varepsilon\gamma_n^{-1}D_{f_n}(x_{n+1},x_n)
\leq\theta(x)+\gamma_n^{-1}D_{f_n}(x,x_n).
\end{equation}
In particular, since $x_n\in C$,
\begin{equation}
\label{e:7616}
\theta(x_{n+1})+\gamma_n^{-1}\big(D_{f_n}(x_n,x_{n+1})
+\varepsilon D_{f_n}(x_{n+1},x_n)\big)
\leq\theta(x_n).
\end{equation}
This shows that 
\begin{equation}
\label{e:7902}
\big(\theta(x_n)\big)_{n\in\NN}\;\text{decreases.}
\end{equation}
In turn, using the coercivity of $\theta$,
we infer that $(x_n)_{n\in\NN}$ is bounded, which secures
\ref{t:1a} in Theorem~\ref{t:1}. It remains to verify
that Algorithm~\ref{a:2} is focusing.
Towards this end, let $z\in\mathscr{S}$ and suppose that
\begin{equation}
\label{e:7980}
\big(D_{f_n}(z,x_n)\big)_{n\in\NN}\;\text{converges}
\end{equation}
and
\begin{equation}
\label{e:7372}
\varepsilon\sum_{n\in\NN}D_{f_n}(x_{n+1},x_n)
\leq\sum_{n\in\NN}(1-\kappa\gamma_n/\alpha)
D_{f_n}(x_{n+1},x_n)<\pinf.
\end{equation}
Set $\gamma=\inf_{n\in\NN}\gamma_n$ and $\ell=\lim D_{f_n}(z,x_n)$.
It follows from \eqref{e:7610} applied to $z\in C$ that
\begin{equation}
\label{e:7613}
(\forall n\in\NN)\quad
\gamma\big(\theta(x_{n+1})-\min\theta(\XX)\big)+
D_{f_n}(z,x_{n+1})+\varepsilon D_{f_n}(x_{n+1},x_n)\leq
D_{f_n}(z,x_n)
\end{equation}
and therefore from condition~\ref{a:2b} in 
Algorithm~\ref{a:2} that
\begin{align}
\label{e:9641}
&(\forall n\in\NN)\quad
\gamma\big(\theta(x_{n+1})
-\min\theta(\XX)\big)
+D_{f_{n+1}}(z,x_{n+1})
+\varepsilon D_{f_n}(x_{n+1},x_n)
\nonumber\\
&\hskip 24mm
\leq(1+\eta_n)\Big(
\gamma\big(\theta(x_{n+1})-\min\theta(\XX)\big)
+D_{f_n}(z,x_{n+1})+\varepsilon D_{f_n}(x_{n+1},x_n)\Big)
\nonumber\\
&\hskip 24mm
\leq(1+\eta_n)D_{f_n}(z,x_n).
\end{align}
Hence, 
$\varlimsup\gamma(\theta(x_{n+1})-\min\theta(\XX))+\ell\leq\ell$
and therefore $\varlimsup(\theta(x_{n+1})-\min\theta(\XX))=0$.
Thus 
\begin{equation}
\label{e:k1940}
\theta(x_n)\to\min\theta(\XX).
\end{equation}
Now take $x\in\WC(x_n)_{n\in\NN}$, say $x_{k_n}\weakly x$.
By weak lower semicontinuity of $\theta$,
$\min\theta(\XX)\leq\theta(x)
\leq\varliminf\theta(x_{k_n})=\min\theta(\XX)$
and it follows that $x\in\Argmin\theta=\zer(A+B)$.
Consequently, Theorem~\ref{t:1} asserts that $(x_n)_{n\in\NN}$
converges weakly to a point in $\mathscr{S}$.

\ref{t:3ii-}: Combine \eqref{e:7902} and \eqref{e:k1940}.

\ref{t:3ii}\&\ref{t:3iii}:
Fix $z\in\mathscr{S}$ and set $\gamma=\inf_{n\in\NN}\gamma_n$. 
Arguing along the same lines as above, we obtain
\begin{equation}
\label{e:6164}
(\forall n\in\NN)\quad
\gamma\big(\theta(x_{n+1})-\min\theta(\XX)\big)
+D_{f_{n+1}}(z,x_{n+1})+\varepsilon D_{f_n}(x_{n+1},x_n)
\leq(1+\eta_n)D_{f_n}(z,x_n)
\end{equation}
and therefore \cite[Lemma~5.31]{Livre1} guarantees that
$\sum_{n\in\NN}(\theta(x_n)-\min\theta(\XX))<\pinf$.
In addition, $(\theta(x_n)-\min\theta(\XX))_{n\in\NN}$
is decreasing by virtue of \eqref{e:7902}. However, recall that
if $(\alpha_n)_{n\in\NN}$ is a decreasing sequence in $\RP$
such that $\sum_{n\in\NN}\alpha_n<\pinf$, then
\begin{equation}
\alpha_n=o\bigg(\frac{1}{n}\bigg)\quad\text{and}\quad
\Sum_{n\in\NN}n(\alpha_n-\alpha_{n+1})<\pinf.
\end{equation}
Hence, $\theta(x_n)-\min\theta(\XX)=o(1/n)$
and $\sum_{n\in\NN}n(\theta(x_n)-\theta(x_{n+1}))<\pinf$.
Consequently, since \eqref{e:7610} yields
\begin{equation}
(\forall n\in\NN)\quad\gamma_n^{-1}D_{f_n}(x_n,x_{n+1})
+\varepsilon\gamma_n^{-1}D_{f_n}(x_{n+1},x_n)
\leq\theta(x_n)-\theta(x_{n+1}),
\end{equation}
we infer that
$\sum_{n\in\NN}n(D_{f_n}(x_{n+1},x_n)+D_{f_n}(x_n,x_{n+1}))<\pinf$.
\end{proof}

\begin{remark}
\label{r:12}
Let us relate Theorem~\ref{t:3} to the literature.
\begin{enumerate}
\item
\label{r:12i}
The conclusions of items~\ref{t:3i} and \ref{t:3ii-}
are obtained in
\cite[Theorem~1(2)]{Nguy17} under
more restrictive conditions on the sequences 
$(\gamma_n)_{n\in\NN}$ and $(f_n)_{n\in\NN}$. Thus,
we do not require in Theorem~\ref{t:3} the additional
condition $(\forall n\in\NN)$
$(1+\eta_n)\gamma_n-\gamma_{n+1}\leq\alpha\eta_n/\kappa$. 
Furthermore, we do not suppose either that 
${-}\ran\nabla\psi\subset\dom\varphi^*$ or that the functions
$(f_n)_{n\in\NN}$ are cofinite. 
\item
\label{r:12ii}
Items \ref{t:3ii} and \ref{t:3iii} are new even in
Euclidean spaces. In the finite-dimensional setting,
partial results can be found in \cite{Baus17}, where:
\begin{enumerate}
\item
A single convex function is used: $(\forall n\in\NN)$ $f_n=f$.
\item
The viability of the sequence $(x_n)_{n\in\NN}$ is a blanket
assumption, while it is guaranteed in Theorem~\ref{t:3}.
\item
Only the rates $\sum_{n\in\NN}D_{f}(x_{n+1},x_n)<\pinf$ and
$(\varphi+\psi)(x_n)-\min(\varphi+\psi)(\XX)=O(1/n)$ are obtained.
\end{enumerate}
\end{enumerate}
\end{remark}

\subsection{Further applications}
\label{sec:quang}

Theorems~\ref{t:1} and \ref{t:3} operate under broad assumptions
which go beyond those of the existing forward-backward methods of
\cite{Sico03,Opti14,Nguy17,Rena97} described in 
\eqref{e:5}--\eqref{e:8}. Here are two examples which do not fit
the existing scenarios and exploit this generality. 

\begin{example}
\label{ex:61}
Consider the setting of Problem~\ref{prob:1}.
Suppose, in addition, that the following hold:
\begin{enumerate}[label={\rm[\alph*]}]
\item
$A$ is uniformly monotone on bounded sets.
\item
There exist $\psi\in\Gamma_0(\XX)$ and $\kappa\in\RPP$
such that $B=\partial\psi$ and $(\forall x\in C)(\forall y\in C)$ 
$D_\psi(x,y)\leq\kappa D_f(x,y)$.
\item
$f$ is supercoercive.
\item
$\zer(A+B)\subset\intdom f$.
\end{enumerate}
Let $(\gamma_n)_{n\in\NN}$ be a sequence in $\RPP$ such that
$0<\inf_{n\in\NN}\gamma_n\leq\sup_{n\in\NN}\gamma_n<1/\kappa$,
take $x_0\in C$, and set $(\forall n\in\NN)$ $x_{n+1}=
(\nabla f+\gamma_n A)^{-1}(\nabla f(x_n)-\gamma_n \nabla\psi(x_n))$.
Then $(x_n)_{n\in\NN}$ converges strongly to the unique zero of
$A+\nabla\psi$.
\end{example}

The next example concerns variational inequalities.

\begin{example}
\label{ex:12}
Let $\varphi\in\Gamma_0(\XX)$, let $B\colon\XX\to 2^{\XX^*}$
be maximally monotone, let $f\in\Gamma_0(\XX)$ be 
essentially smooth, and set 
$C=(\intdom f)\cap\dom\partial\varphi$.
Suppose that $C\subset\intdom B$ and $B$ is single-valued on
$\intdom B$. Consider the problem of finding a point in 
\begin{equation}
\mathscr{S}=\menge{x\in C}{(\forall y\in\XX)\;
\pair{x-y}{Bx}+\varphi(x)\leq\varphi(y)},
\end{equation}
which is assumed to be nonempty. This is a special case of
Problem~\ref{prob:1} with $A=\partial\varphi$ and,
given $x_0\in C$, Algorithm~\ref{a:1} produces the iterations
$(\forall n\in\NN)$ $x_{n+1}=\prox_{\gamma_n\varphi}^{f_n}
(\nabla f_n(x_n)-\gamma_nBx_n)$. The weak convergence of
$(x_n)_{n\in\NN}$ to a point in $\mathscr{S}$ is discussed in
Theorem~\ref{t:1}. Even in Euclidean spaces, this scheme is new and
of interest since, as shown in \cite{Baus17,Joca16,Nguy17}, the 
Bregman proximity operator $\prox_{\gamma_n\varphi}^{f_n}$ may be
easier to compute for a particular $f_n$ than for the standard
kernel $\|\mute\|^2/2$. Altogether, our framework makes it
possible to solve variational inequalities by forward-backward
splitting with non-cocoercive operators and/or outside of Hilbert
spaces. 
\end{example}

\end{document}